\documentclass[11pt]{article}
\usepackage{fullpage}

\usepackage{amsmath,amsthm,amsfonts,amssymb,bm,wasysym}
\usepackage[usenames]{color}
\usepackage{hyperref}
\usepackage{comment}
\usepackage{bbm}
\usepackage{enumerate} 
\usepackage[utf8]{inputenc}
\usepackage[dvipsnames]{xcolor}
\usepackage{breakcites}

\theoremstyle{plain}  
	\newtheorem{theorem}{Theorem}
		\newtheorem{proposition}[theorem]{Proposition}
		\newtheorem{lemma}[theorem]{Lemma}
		\newtheorem{corollary}[theorem]{Corollary}
		\newtheorem{claim}{Claim}
		\newtheorem{assumption}[theorem]{Assumption}

\theoremstyle{definition}

\theoremstyle{definition}

		\newtheorem{exampleEnv}[theorem]{Example}
						{\pushQED{\qed}\exampleEnv}
						{\popQED\endexampleEnv}
				\newenvironment{example*}{\exampleEnv}{\endexampleEnv}
		\newtheorem{remarkEnv}[theorem]{Remark}
				\newenvironment{remark}  % unstarred version has \qed
						{\pushQED{\qed}\remarkEnv}
						{\popQED\endremarkEnv}
				\newenvironment{remark*}{\remarkEnv}{\endremarkEnv}

\usepackage{titlesec}
\titleformat*{\section}{\Large\centering\normalfont\bfseries}
\titleformat{\subsection}[runin]
{\normalfont\large\bfseries}{\thesubsection}{1em}{}
\titleformat{\subsubsection}[runin]
{\normalfont\bfseries}{\thesubsubsection}{1em}{}

\usepackage{blindtext}
\usepackage{hyperref}
\def\R{\mathbb{R}} 
\def\Sd{{\mathbb S}^{d-1}}

\newcommand{\tr}[1]{ \mathop\mathrm{tr} \left(#1\right)}

\newcommand{\KL}{{\rm KL}} 
\newcommand{\rs}{r(\Sigma)}

\newcommand{\Na}{\mathbb{N}}
\newcommand{\N}{\mathcal{N}}
\newcommand{\Gv}{\Gamma_{v,\gamma}}
\renewcommand{\Pr}[1]{\mathbb{P}\left(#1\right)} 
\newcommand{\Ex}[1]{\mathbb{E}\left[#1\right]} 
\newcommand{\Ind}[1]{{\bf 1}_{#1}}
\newcommand{\Var}[1]{{\rm Var}\left(#1\right)} 
\newcommand\norm[1]{\left\lVert #1\right\rVert}
\newcommand\abs[1]{\left| #1\right|}
\newcommand\normop[1]{\left\lVert #1\right\rVert_{\rm op} }
\newcommand\inq[2]{\langle #1, #2 \rangle^4}
\newcommand\insq[2]{\langle #1, #2 \rangle^2}
\newcommand\inner[2]{\langle #1, #2 \rangle}
\newcommand\ip[2]{\langle #1, #2 \rangle}

\newcommand{\Est}{\widehat{{\sf E}}}
\newcommand{\Trest}{\widehat{{\sf T}}}
\newcommand{\est}{\widehat{{\sf e}}}
\numberwithin{theorem}{section}
\numberwithin{equation}{section}

\definecolor{Quase}{RGB}{144,26,30}

\begin{document}
\title{Improved covariance estimation: optimal robustness and sub-Gaussian guarantees under heavy tails}
\author{Roberto I. Oliveira\thanks{IMPA, Rio de Janeiro, Brazil. \texttt{rimfo@impa.br}. Research supported by ``Bolsa de Produtividade em Pesquisa'' and a ``Projeto Universal'' (432310/2018-5) from CNPq, Brazil; and by a ``Cientista do Nosso Estado'' grant (E26/200.485/2023) from FAPERJ, Rio de Janeiro, Brazil.}~ and Zoraida F. Rico\thanks{Columbia University, New York. \texttt{zoraida.f.rico@columbia.edu.}}}
\date{}
\maketitle
\begin{abstract} We present an estimator of the covariance matrix $\Sigma$ of random $d$-dimensional vector from an i.i.d. sample of size $n$. Our sole assumption is that this vector satisfies a bounded $L^p-L^2$ moment assumption over its one-dimensional marginals, for some $p\geq 4$. Given this, we show that $\Sigma$ can be estimated from the sample with the same high-probability error rates that the sample covariance matrix achieves in the case of Gaussian data. This holds even though we allow for very general distributions that may not have moments of order $>p$. Moreover, our estimator can be made to be optimally robust to adversarial contamination. This result improves the recent contributions by Mendelson and Zhivotovskiy and Catoni and Giulini, and matches parallel work by Abdalla and Zhivotovskiy (the exact relationship with this last work is described in the paper).\end{abstract}

\section{Introduction}\label{sec:introCov}
The covariance matrix estimation is a classical problem in Multivariate Statistics. In this paper, we are interested in estimating the covariance matrix $\Sigma = \Ex{XX^\top} $ of a random vector $X \in \R^d$ with
zero mean from independent and identically distributed (i.i.d.) copies $X_1, \dots, X_n$ of $X$. We study this problem under (relatively) heavy tails, from a nonasymptotic perspective. This less classical setting has received much recent attention \cite{mendelson2014singular,tikhomirov2018sample,minsker2018sub}.
\\
A natural way to address the problem is using the sample covariance matrix, defined as \[\widehat{\Sigma}_n=\frac{1}{n}\sum_{i=1}^n(X_i)(X_i)^\top.\] This research has yielded important results in obtaining concentration inequalities of the sample covariance's deviation from the true covariance in the operator norm. This approach covers both classical asymptotic work and more recent nonasymptotic bounds, as discussed in \cite{mendelson2014singular,tikhomirov2018sample,vershynin2011introduction}. The dependence of the dimension in the problem is a central focus in this research. Remarkably, other papers such as \cite{lounici, kolt} have achieved dimension-free results. In \cite{lounici}, the author derives bounds on the operator norm in terms of the “effective rank of the covariance matrix” denoted by \[\rs=\frac{\tr{\Sigma}}{\normop{\Sigma}}.\] Meanwhile, in \cite{kolt}, the authors employ a chaining-based approach to derive concentration inequalities and expectation bounds for the operator norm of centered Gaussian random variables in a separable Banach space. Notably, these bounds are independent of the dimension of the space. Furthermore, a recent preprint by Zhivotovskiy \cite{zhivotovskiy2021dimension} presents related dimension-free bounds, assuming light-tail conditions on the vector $X$. 
\\
Another direction of analysis is designing a covariance matrix estimator with tight operator norm deviation bounds, while making minimal assumptions about the distribution of $X$. Results in this direction include  \cite{minsker2018sub,minskrobust,ostrovskii2019affine}. The current state-of-the-art regarding probability bounds is in \cite{SharZhi} and \cite{catoni2017}. Mendelson and Zhivotovskiy demonstrate that there exists a covariance estimator that achieves, up to logarithmic factors, a rate of error that matches the error of the sample covariance in the Gaussian case. Remarkably, this holds under only the assumption of bounded kurtosis on the one dimensional marginals of the data distribution. According to \cite{abdalla2022}, Catoni and Giulini \cite{catoni2017} find better bounds without the log factors in the same setting. However, it is quite hard to extract this result from their very sophisticated manuscript, and it seems that optimal results can only be obtained with some knowledge of distributional parameters. We will have more to say about these papers in \S \ref{sub:recent}.

This paper presents a theoretically simpler approach to covariance estimation with heavy tails and contamination. Our estimator is based on a high-dimensional ``trimmed mean'' idea that is optimally robust against adversarial corruption of the data. Even in the non-contaminated case, it removes the spurious $\log \rs$ factors in \cite{SharZhi}, and matches the results of \cite{catoni2017} under fourth-moment conditions, without requiring any knowledge of distribution parameters.

\begin{remark}[Parallel work by Abdalla and Zhivotovskiy]\label{rem:parallel} The original version of this paper comes from the PhD thesis of the second-named author \cite[Chapter 3]{rico2022}, which was defended on April 29th, 2022. It was somewhat weaker than the present result, in that it only worked under a fourth-moment assumption on the data. As the final version of the thesis was being submitted, Abdalla and Zhivotovskiy \cite{abdalla2022} posted a preprint with results under $p$-th moment assumptions, for arbitrary $p\geq 4$. The works of our two groups were completely independent up to that point. After \cite{abdalla2022} was posted, those authors encouraged us to revise our proof in light of our present results, especially because their own proof for $p>4$ was quite convoluted. In revising our arguments, we found a significantly streamlined version of the proof in \cite{rico2022} that works for all $p\geq 4$. We nevertheless claim that the core of the argument in \cite{rico2022} is preserved. The similarities and differences between \cite{abdalla2022} and the present paper will be discussed at several points of this introduction. We wish to thank Abdalla and Zhivotovskiy right away for their stimulating discussions on this topic.\end{remark} 

\subsection{The main assumption} To state our main result, we first present the assumption we will make throughout the paper. Undefined notation may be found in Section \ref{sec:prelim}.

\begin{assumption}[i.i.d. sample with contamination]\label{assum1}$X_1,\dots,X_n$ are i.i.d. copies of a random element $X$ of $\R^d$ satisfying $\Ex{\|X\|^p}<+\infty$ for some $p\geq 4$. We assume $\Ex{X}=0$, that the covariance $\Sigma$ of $X$ is non-null, and set 
\[\kappa_p:=\sup\limits_{v\in\R^d,\,\inner{v}{\Sigma v}=1}\Ex{|\inner{X}{v}|^p}^{\frac{1}{p}}.\]
Finally, we let $Y_1,\dots,Y_n$ denote other random elements of $\R^d$ that satisfy the following condition:
\[\#\{i\in [n]\,:Y_i\neq X_i\}\leq \eta n\]
for some $\eta\in [0,1)$.\end{assumption}

Notice the Assumption only requires finitely many moments of the vector $X$. Therefore, it allows for relatively heavy tails. The sole requirement is that the one-dimensional marginals satisfy a moment condition, $\kappa_p^2$ is referred to as the $L^p-L^2$ hypercontractivity constant by \cite{abdalla2022} and elsewhere in the literature. We highlight a few properties of this moment assumption.

\begin{enumerate}
\item If the vector $X=(X(1),X(2),\dots,X(d))^\top$ has independent coordinates, then $\kappa_p$ can be bounded by $C_p\,\max_{1\leq i\leq d}\|X(i)\|_{L^p}/\|X(i)\|_{L^2}$, where $C_p>0$ depends on $p$ only. 
\item The same bound holds if $X$ has a distribution that is {\em unconditional}, that is, if the law of $X$ is invariant by arbitrary sign changes of its coordinates. 
\item If $X$ satisfies $\kappa_p<+\infty$, the same holds for any linear transformations of $X$ (with the same $\kappa_p$). 
\item Consider the scenario where $X$ follows an {\em elliptical distribution} with zero mean. In such case, there exist a unconditional random vector $\tilde{X}$ (in the sense of item 2) and a linear transformation $T$ such that $X=T\tilde{X}$. Then, the constant $\kappa_p$ for $X$ is the same as that of $\tilde{X}$.
\end{enumerate}

Assumption \ref{assum1} also includes the possibility of {\em sample contamination}. That is, all estimators we consider will be computed on a random sample, denoted as $Y_1, \ldots, Y_n$, which may differ from the original sample $X_1, \ldots, X_n$ by at most $\eta n$ indices. This concept aligns with the ``$\epsilon$-replacement''~notion as defined in \cite{donoho1983notion} and further popularized as the ``adversarial corruption''~model, extensively studied in recent literature \cite{diakonikolas2019recent, diakorobust, lugosi2021}. We aim to develop estimators that can withstand this type of contamination, which is is a much more demanding model than Huber contamination \cite{huber1981}. 

Lastly, we note that the requirement $\Ex{X}=0$ is introduced solely for convenience. In practice (assuming $n$ is even for simplicity), when $\Ex{X}\neq 0$, one can consider the vectors $\bar{X}_i:=(X_{2i-1}-X_{2i})/\sqrt{2}$ and $\bar{Y}_i:=(Y_{2i-1}-Y_{2i})/\sqrt{2}$ for $i=1,\dots,n/2$. These vectors now satisfy Assumption \ref{assum1} with a new sample size $n/2$, a new moment constant $\bar{\kappa}_p = \sqrt{2}\kappa_p$ and a new contamination level $\bar{\eta}=2\eta$.

\subsection{The main result} The next theorem is the key contribution of the present paper. In what follows, $\R^{d\times d}_{\geq 0}$ is the set of $d\times d$ symmetric positive semidefinite matrices, and $\normop{\cdot}$ is the operator norm.

\begin{theorem}[Main result; proof in \S \ref{sub:finalfinal}]\label{thm:main} There exists a constant $C>0$ such that the following holds. Fix a confidence parameter $1-\alpha\in (0,1)$, a sample size $n\in\Na$ and a contamination parameter $\eta\in [0,1/2)$. Then, there is an estimator (i.e. a measurable function) $\Est_{\star}:(\R^{d})^n \to \R^{d\times d}_{\geq 0}$, depending on $\alpha$, $\eta$, and $n$, such that, whenever Assumption \ref{assum1} is satisfied, and additionally $\eta\leq 1/C\kappa_4^4$ and $n\geq C\,(\rs + \log(2/\alpha)),$ the following holds with probability $\geq 1-\alpha$:
\[\normop{\Est_{\star}(Y_1,\dots,Y_n) - \Sigma}\leq C\kappa^2_4\normop{\Sigma}\left( \sqrt{\frac{\rs}{n}} + \sqrt{\frac{\log(2/\alpha)}{n}}\right) + C\kappa_p^2\normop{\Sigma}\eta^{1-\frac{2}{p}}.\]\end{theorem}

Ignoring the $\kappa_4$ factor, the error bound we obtain in estimating $\Sigma$ consists of three terms.

\begin{enumerate}
\item One term of order $\normop{\Sigma}\sqrt{\rs/n}$, which cannot be avoided in a minimax sense (see \cite[Section 5]{abdalla2022} and \cite[Theorem 2]{lounici});
\item The second term $\normop{\Sigma}\sqrt{\log(2/\alpha)/n}$ is also unavoidable, as shown in \cite[Section 5]{abdalla2022}.
\item Finally, the term proportional to $\kappa_p^2\normop{\Sigma}\eta^{1-2/p}$ is minimax-optimal for $p=4$, and also in the sub-Gaussian case where $\kappa_p=O(\sqrt{p})$ for large $p$; see again \cite[Section 5]{abdalla2022} for details.
\end{enumerate}

As mentioned in Remark \ref{rem:parallel}, our first version of Theorem \ref{thm:main}, dating from April 2022, only considered the case $p=4$, and had a more complicated proof. The present form and proof of Theorem \ref{thm:main} emerged after an interaction with the authors of \cite{abdalla2022}. As we explain below, the proof strategy in the present paper is a streamlining of our original argument. We will also see in \S \ref{sub:recent} that certain special cases of Theorem \ref{thm:main} could have been deduced from earlier work \cite{catoni2017,zhivotovskiy2021dimension}.

In the context of relaxing weak moment assumptions, \S \ref{appendixD} in the Appendix introduces a variant of our main result, Theorem \ref{thm:main}. This variant eases the moment conditions related to $L^q-L^2$ norm equivalence (hypercontractivity) with $2\leq q\leq 4$ and delivers bounds of the following form: \[\kappa_q^2\normop{\Sigma}\left(\frac{\rs + \log(2/\alpha)}{n}\right)^{1-\frac{2}{q}}+\kappa_q^2\normop{\Sigma}\eta^{1-\frac{2}{q}}.\]

\subsection{Some details on recent papers on covariance estimation}\label{sub:recent} Before we outline our approach to Theorem \ref{thm:main}, we briefly describe the relevant recent work on the same topic. One point to mention is that neither this work, nor the papers we consider obtain computationally efficient estimators. The focus is on information-theoretically optimal results. 

The best published result on finite-sample covariance estimation is the one by Mendelson and Zhivotovskiy \cite{SharZhi}. In the setting of Theorem \ref{thm:main}, their estimator $\widehat{\Sigma}$ works for $\eta=0$ (no contamination), requires a sample size $n\geq C\,(\rs \log \rs + \log(2/\alpha))$, and achieves the following bound with probability $\geq 1-\alpha$:
\[\normop{\widehat{\Sigma} - \Sigma}\leq C\kappa_4^2\normop{\Sigma}\left( \sqrt{\frac{\rs\log\rs}{n}} + \sqrt{\frac{\log(1/\alpha)}{n}}\right).\]
The presence of the additional $\log \rs$ factors come from their reliance on matrix concentration inequalities by Minsker \cite{minsker2017bernstein} and  Tropp \cite{tropp2015matrixconc}, which unavoidably feature such $\log\rs$ factors. More recent matrix concentration results by Brailovskaya and van Handel \cite[Section 8]{brailovskaya2022universality} contain dimension-dependent factors in their error terms. All in all, it seems that no proof strategy based on off-the-shelf matrix concentration will give us an analogue of Theorem \ref{thm:main}, even with $\eta=0$. 

As mentioned above, the result obtained by Catoni and Giulini \cite{catoni2017} can be employed to recover the bound of Theorem \ref{thm:main} when $\eta=0$. Their estimator is based on a very general PAC-Bayesian framework which relies on certain ``influence functions.''  It seems to us that their approach requires previous knowledge of certain problem parameters, such as what we refer to as $\kappa_4$, to achieve optimal guarantees. Further details can be found in the discussion preceding Proposition 4.2 in \cite{catoni2017}. 

Zhivotovskiy \cite[Lemma 5]{zhivotovskiy2021dimension} also (implicitly) gives a version of Theorem \ref{thm:main} in the contamination-free setting, at least when $\kappa_4<+\infty$ is known. This result was sharpened by Abdalla and Zhivotovskiy \cite{abdalla2022}, whose parallel work we have been discussing since Remark \ref{rem:parallel}. The approach in \cite{abdalla2022} combines different ideas. One of them comes from the literature of sample covariance: a line of works culminating in \cite{tikhomirov2018sample} have identified that the sum 
\[\frac{1}{n}\sum_{i=1}^n\inner{X_i}{v}^2\]
contains a small number of ``peaky''~terms (with relatively large magnitude) and a large number of ``spread'' ~terms. Their initial contribution involves obtaining a dimension-free formulation of \cite{tikhomirov2018sample}. They then obtain their version of Theorem \ref{thm:main} by using an influence-function (or ``soft truncation'') approach combined with PAC-Bayesian techniques, in the same vein as \cite{zhivotovskiy2021dimension}. This leads to significant technical difficulties, which we manage to bypass with our analysis. One of our simplifications is to notice that PAC-Bayesian arguments give us direct control over some uniform counting events. This avoids using the arguments in \cite{abdalla2022} that involve Bai-Yin-type bounds. See Remark \ref{rem:othercountingresults} for more details on this point.

Finally, Mendelson \cite{mendelson2021lp} considered the more general problem of estimating the $L^p$ moments of one dimensional marginals of a random vector $X$. For this purpose, \cite{mendelson2021lp} develops an approach via VC dimension arguments. In the case of covariance estimation ($p=2$), this leads to dimension-dependent bounds. For example, in the contamination-free setting, and under a $L^4-L^2$ assumption, \cite{mendelson2021lp} obtains error bounds of the order of $\sqrt{d\log(en/d)/n}$, with probability $1-e^{-cd}$ (for some $c>0$). However, the estimator in \cite{mendelson2021lp} is essentially the same as ours, which we now describe. 

\subsection{Our approach via sample trimming}\label{sub:approach} The proof of Theorem \ref{thm:main} can be outlined as follows. 

\begin{enumerate}
    \item Focus on estimating $\inner{v}{\Sigma v}$ uniformly over vectors $v$ in the unit sphere $\Sd$, which is equivalent to estimating $\Sigma$. \item Consider the following {\em trimmed mean} estimators for $\inner{v}{\Sigma v}$: 
    \[\est_k(v):=\frac{1}{n-k}\inf\limits_{S\subset [n],\# S =n-k}\sum_{i\in S}\inner{Y_i}{v}^2.\]
    \item Show the following deterministic result. Under a {\em counting condition},\[\#\{i\in [n]\,:\, \inner{X_i}{v}^2>B\}\leq t,\] 
    we find the following approximation (for $k\approx \eta n + t$):
    \[\sup_{v\in\Sd}|\est_k(v) - \inner{v}{\Sigma v}|\approx \varepsilon(B):=\sup_{v\in\Sd}\left|\frac{1}{n}\sum_{i=1}^n \inner{X_i}{v}^2\wedge B - \Ex{\inner{X_1}{v}^2\wedge B}\right|.\]
  We refer to $\varepsilon(B)$ as the truncated empirical process.
    \item Use PAC-Bayesian techniques to analyze the counting condition and the truncated empirical process. 
    \item From this analysis, show that the estimator is good for a range of values $k$. 
    \item Choose a ``good value''~$\widehat{k}$ of the trimming parameter that is good with high probability, and output $\est_{\widehat{k}}(v)$ for each $v\in\Sd$, corresponding to our final estimator $\Est_{\star}$.
\end{enumerate}

\textit{Step 1} -- described in Proposition \ref{prop:estimatorisdefined} below -- is simple, but leads to computationally inefficient estimation. See Remark \ref{rem:computational} for details.  

\textit{Step 2} is also natural, as the rationale for trimmed means is that they are naturally robust against contamination. An adequate choice of $k$ (step 6) will be crucial for the final estimator: this will require estimating the effective rank $\rs$.   

\textit{Step 3} is a simple, but crucial contribution. A more complicated version of this fact is implicit in previous work on mean estimation for vectors by Lugosi and Mendelson \cite{lugosi2021}, as well as in \cite[Lemma 5]{abdalla2022}. However, we view it as a fundamental result. Details on this step are given in Section \ref{sec:overview}. Here, we note that the PAC-Bayesian methods is particularly efficient in obtaining counting conditions; see Remark \ref{rem:othercountingresults} for details. On the other hand, we also note that this approach loses large constant factors, and better alternatives are available for scalar mean estimation \cite[Chapter 2]{rico2022}.  

\textit{Step 4} is the technical core of the paper. We use a novel PAC-Bayesian version of the Bernstein concentration inequality for bounded random variables (Theorem \ref{th:PacBernstein}). Let us pause to analyze this step in detail.

At a high level, both the counting condition and the truncated empirical process lead us to analyze suprema of families of random variables $\{Z(\theta)\}_{\theta\in \Theta}$, where $\Theta$ is some space of parameters. The most common technique for bounding such processes is {\em chaining}, i.e. performing discretizations of $\Theta$ and strong bounds on the ``increments'' $Z(\theta)-Z(\theta')$. In the matrix setting, other methods based on matrix inequalities are also available \cite{tropp2015matrixconc,minsker2017bernstein}. 

PAC-Bayesian methods are an alternative to these approaches. Suppose we consider a {\em smoothed} version of the process: that is, instead of considering the random variables $Z(\theta)$ directly, we look at averaged versions of the form $\{\int_{\Theta}Z(\theta)\mu(d\theta)\}_{\mu\in\mathcal{M}}$, where $\mathcal{M}$ is a suitable family of probability measures over $\Theta$. PAC-Bayesian methods allow us to control the supremum of this smoothed process via bounds on the moment generating function of $Z(\theta)$ at single points, plus a price related to KL divergences in $\mathcal{M}$. Bounding the original process is then a matter of quantifying the difference between smoothed and unsmoothed processes. 

In this paper, we follow previous work \cite{zhivotovskiy2021dimension,catoni2016pac} and apply the PAC-Bayesian method with Gaussian smoothing. Specifically, Section \ref{sec:countingforvectors} obtains a probability bound on the counting condition. As we note in Remark \ref{rem:othercountingresults}, our approach is somewhat different from (and perhaps simpler than) a related technical lemma by Abdalla and Zhivotovskiy \cite[Lemma 4.2]{abdalla2022}. Section \ref{sec:vectors} then analyses the truncated empirical process via comparison with the smoothed process. In this second case, the calculations needed to bound the effect of smoothing are quite technical, but the general ideas are fairly straightforward, and lead naturally to steps 5 and 6 above. 

We finish this subsection with several remarks. 

\begin{remark}[Comparison with Mendelson \cite{mendelson2021lp}] We emphasize that this kind of estimator was previously considered in \cite{mendelson2021lp}. However, our proof techniques are quite different, and allow us to obtain dimension-free results. Another difference between the two papers lies in how to choose the trimming parameter $k$. \end{remark}

\begin{remark}[Trimming vs. soft truncation \cite{abdalla2022}] As noted, the estimator by Abdalla and Zhivotovskiy \cite{abdalla2022} relies on a ``soft truncation'' of the sample, where the soft truncation level has to be chosen in a data-dependent fashion. Our arguments in this paper suggest that trimming the sample is essentially equivalent to performing truncation. In our case, the right trimming parameter $\widehat{k}$ must be chosen from the data, just like the truncation level in \cite{abdalla2022}.\end{remark}

\begin{remark}[Further remarks on trimming] The approach outlined above is quite general and (we believe) natural. In \cite{oliveira2023trimmed}, Resende and the first-named author have used similar ideas to control general ``trimmed empirical processes.'' This improves the best-known mean estimators for vectors under general norms \cite{depersin2021general,lugosi19}, but does not seem to give back the results of the present paper.\end{remark}

\begin{remark}[On a previous version of Theorem \ref{thm:main}] We come back to the previous version of the present results in Rico's thesis \cite[Chapter 3]{rico2022}. This required a kind of ``convexified''~trimmed mean in the spirit of weight-based algorithms for robust mean estimation \cite{diakonikolas2019recent,diakorobust,hopkins2020robust}. In particular, it required a version of the truncated empirical process indexed by ``density matrices.'' Although the thesis focused on $p=4$, it is clear from that analysis that dealing with other $p$ only requires a different analysis of the counting condition, which is not hard. What was less obvious is that the ``convexification''~could be avoided altogether, as we have found after the suggestion by Abdalla and Zhivotovskiy. Still, we think the convexified approach may lead to efficient algorithms.\end{remark}

\subsection{Further background} We now include additional references to related work that was not previously discussed in detail.

The present paper belongs to a line of research that consists of estimating means and covariances of distributions with best-possible nonasymptotic performance. Although the sample mean is the  asymptotically optimal estimator in one dimension, Catoni's seminal paper \cite{catoni2012challenging} showed that it can be greatly improved in finite-sample settings with known variance. More specifically, that paper shows that Chebyshev's inequality is the tight deviation bound for the sample mean (up to constants), but there are other estimators achieving Gaussian-like behavior. So-called ``sub-Gaussian mean estimators'' in one dimension were further studied in \cite{dllo2016}. 

The literature soon moved to the estimation of means of vectors. Minsker \cite{minsker2015geometric} provided a a general ``geometric median~estimator'' for random vectors in a Banach space, with good (but suboptimal) finite-sample properties. After preliminary results by Joly et al. \cite{meanMult}, Lugosi and Mendelson were the first to obtain a sub-Gaussian estimator for vectors in $\R^d$ with the Euclidean norm \cite{lugosi2019}. Further results in this are include refinements of the Euclidean estimator \cite{lugosi2021}; computationally efficient algorithms, implementing the original Lugosi-Mendelson construction e.g. \cite{hopkins2020mean,depersin2022mean}; and nearly optimal estimators for general norms \cite{depersin2021general}. 

Estimating means of matrices (including covariance matrices) from a random sample is a particular case of mean estimation under general norms. However, the best results in that problem seem to come from approaches that are specific to matrices. The important works of Catoni and Giulini \cite{catoni2016pac} and \cite{catoni2017} use PAC-Bayesian methods estimate vectors and covariance matrices; the second of these papers was commented on above. Minsker's paper \cite{minsker2018sub} works for general matrices, but loses logarithmic factors. The aforementioned \cite{SharZhi,minsker2018sub} also deal with matrix estimation in this sub-Gaussian sense. 

We note in passing that there is a growing Computer Science literature on computationally efficient robust estimators. Of the main references in the area, we cite the seminal \cite{diakorobust}, the survey \cite{diakonikolas2019recent}, and the paper by Hopkins et al. \cite{hopkins2020robust} which emphasizes the use of weights on samples. The original approach in \cite[Chapter 3]{rico2022} was based on these ideas, but was not computationally efficient. 

\subsection{Organization} 
The plan of the rest of this article is as follows. The next section presents some preliminaries and our PAC-Bayesian Bernstein inequality. Section \ref{sec:overview} describes our estimator and presents a ``warm-up result''~on scalar mean estimation. Section \ref{sec:countingforvectors} derives the ``counting condition'' discussed above. In section \ref{sec:vectors}, we analyse the truncated empirical process for vectors. Section \ref{sec:final_est} is devoted to the proof of Theorem \ref{thm:main}. The proofs of two technical lemmas are deferred the Appendix.

\section{Some preliminaries} \label{sec:prelim}

\subsection{Notation} The cardinality of a finite set $A$ is denoted by $\# A$. For real numbers $x,y$, $x_+:=\max\{x,0\}$, and $x \land y :=\min\{x,y\}$ and $x \lor y :=\max\{x,y\}$ . For $n \in \Na$, $[n]:=\{i \in\Na : 1\leq i\leq n\}$ is the set of numbers from $1$ to $n$. 

We use the standard euclidean norm $\|\cdot\|$ and inner product $\inner{\cdot}{\cdot\cdot}$ over $\R^d$. 
The unit sphere this space is denoted by $\Sd :=\{u \in \R^d: \norm{u}=1\}$. Letting $\R^{d\times d}$ denote the space of $d\times d$ matrices, we also use $\|\cdot\|$ to denote the operator norm over this space, and $\tr{\cdot}$ denotes the trace. $\R^{d\times d}_{\geq 0}$ is the subset consisting of symmetric positive semidefinite matrices. The stable rank of a non-zero matrix $M \in \R^{d\times d}_{\geq 0}$ is given by:
\[r(M) = \frac{\tr{M}}{\norm{M}}.\]

\subsection{Entropic inequalities for Gaussian measures}\label{sub:entropic} We recall here the definition of the Kullback-Leiber divergence between two probability measures, $\mu_0$ and $\mu_1$ on $\R^d$:
\begin{equation*}
 \KL(\mu_1|\mu_0):=
    \begin{cases}
      \int_{\R^d}\log\left(\frac{d\mu_1}{d\mu_0}(\theta)\right)\mu_1(d\theta)  & \mu_1 \ll \mu_0\\
      +\infty & \text{otherwise.}
    \end{cases}       
\end{equation*}
For our purposes, we will be interested in Gaussian measures. Given $\gamma>0$ and $v\in\R^d$, we let $\Gamma_{v,\gamma}$ denote the Gaussian measure over $\R^d$ with mean $v$ and covariance matrix $\gamma \,I_{d\times d}$. In this case, it is well-known that:
\[ \KL(\Gv|\Gamma_{0,\gamma})=\frac{\|v\|^2}{2\gamma^2}.\]
At the same time, a characterization given by the variational formula \cite{LedouxConcen} implies that for all measurable and $\mu_1$-integrable function $h:\R^d\to\R$,
\[ \mu_1(h)\leq \KL(\mu_1|\mu_0)+\log\left(\mu_0(e^h)\right).
\]
As a consequence, the following {\em variational inequality} yields. If $h:\R^d\to \R$ is measurable and $\Gamma_{v,\gamma}$-integrable for all $v\in \Sd$, then
\begin{equation}\label{eq:gaussianentropic}\sup_{v\in \Sd}\Gamma_{v,\gamma}(h)\leq \frac{\gamma^{-2}}{2} + \log\left(\Gamma_{0,\gamma}e^h\right).\end{equation}

\subsection{PAC-Bayesian Bernstein inequality} \label{sub:GeneralPAC} We now introduce methods based on entropic inequalities to work with truncated empirical process such as the ones we encounter in our analysis. As noted in the Introduction, such methods are greatly indebted to Catoni and his collaborators \cite{catoni2016pac,catoni2017}. See also the recent work by Zhivotovskiy \cite{zhivotovskiy2021dimension} for applications to sample covariance matrices. 

Our inequality will require a somewhat complicated setup that we now describe.  A probability space $(\Omega, \mathcal{F}, \mathbb{P})$ is implicit in our discussion. We consider a family of functions 
\[(\theta,\omega)\in\R^d\times \Omega \mapsto Z_i(\theta,\omega)\in\R\,\,(i\in [n])\] that are $(\mathcal{B}(\R^d)\otimes\mathcal{F}) /\mathcal{B}(\R)$-measurable. We use 
$Z_i(\theta)$ to denote the random variable mapping $\omega\in\Omega$ to $Z_i(\theta,\omega)$. We assume that the random variables $\{Z_i(\theta)\}_{i\in [n]}$ are i.i.d. and integrable for any fixed $\theta\in \R^d$.

Let $\Gamma_{v,\gamma}$ be as in \S \ref{sub:entropic}. In what follows, we assume that the integrals:
\[\Gv Z_i(\theta,\omega):=\int_{\R^d}\,Z_i(\theta,\omega)\,\Gamma_{v,\gamma}(d\theta)\]
are well-defined for all $v\in \R^d$ and depend continuously on $v$. Again, we often have the dependence on $\omega$ implicit in our notation; observe, however, that under our assumptions the maps taking $(v,\omega)$ to $\Gamma_{v,\gamma}Z_i(\theta,\omega)$ are also $(\mathcal{B}(\R^d)\otimes\mathcal{F}) /\mathcal{B}(\R)$-measurable. 

\begin{remark}\label{rem:conventionsmoothing}In general, given a function $h$ of $\theta$ and other parameters, we will use the notation $\Gamma_{v,\gamma}h(\theta)$ to denote the integral of $h$ with respect to the $\Gamma_{v,\gamma}$ measure in the $\theta$ variable. \end{remark}

To obtain concentration bounds, we assume that $Z_i(\theta)-\Ex{Z_i(\theta)}\leq A$ almost surely, where $A>0$ is a constant. We also require that \[
\bar{\mu}_\gamma:= \sup_{v\in\Sd}\Gv\Ex{Z_1(\theta)}\mbox{   and   }
\bar{\sigma}_\gamma^2:= \sup_{v\in\Sd} \Gv\Var{Z_1(\theta)}\]
are well-defined. The assumptions we made on the $Z_i(\theta)$ imply that for each fixed $\theta \in \R^d$ and $\alpha\in (0,1)$:
\[ \Pr{\sum_{i=1}^n Z_i(\theta)- n \Ex{Z_i(\theta)} \geq \sqrt{2\Var{Z_i(\theta)}n\log(1/\alpha)} + \frac{A\log(1/\alpha)}{3}} \leq \alpha.\]
Our next result shows that the supremum of the smoothed process $\{\Gv \sum_{i=1}^n Z_i(\theta)\}_{v\in \Sd}$ satisfies similar inequalities. 
\begin{proposition}[Bernstein-type concentration inequality for Gaussian smoothed process] \label{th:PacBernstein} Under the above setup, the following holds with probability $\geq 1-\alpha$:
\[
\sup_{v\in\Sd}\sum_{i=1}^n\Gv\left( Z_i(\theta)\right) \leq n\bar{\mu}_\gamma+ \bar{\sigma}_\gamma\sqrt{n(\gamma^{-2}+2\,\log(1/\alpha))} +\frac{A \left(\gamma^{-2}+2\,\log(1/\alpha)\right)}{6}.
\]
\end{proposition}
\begin{proof} Let us first observe that $\sup_{v\in\Sd}\Gv(\sum_{i=1}^nZ_i(\theta,\omega))$ is a measurable function of $\omega\in\Omega$ because of the continuity assumption. 

Let consider now $\xi\in(0,3/A)$ and $\theta\in \R^d$, and define  
\[S_\xi(\theta):= \sum_{i=1}^n \xi\left[Z_i(\theta)-\Ex{Z_i(\theta)}-\frac{n\,\xi\,\Var{Z_i(\theta)} }{2\left(1-\xi A/3 \right)}\right]. \]
Then, the \textit{entropic inequality for Gaussians} (\ref{eq:gaussianentropic}) implies
\begin{equation}\label{eq:EIconseqGauss}
\frac{\sup_{v\in\Sd}\Gv S_\xi(\theta)}{\xi} \leq \frac{\frac{\gamma^{-2}}{2}+\log\Gamma_{0,\gamma}e^{S_\xi(\theta)}}{\xi}.
\end{equation}\
Next, we claim that for any fixed $\xi\in(0,3/A):$ \begin{equation}\label{eq:claim1PAC}\Pr{\log\Gamma_{0,\gamma} e^{S_\xi(\theta)} > \log(1/\alpha)} \leq \alpha. 
\end{equation} Indeed, by the Markov's Inequality and Fubini, it follows
\begin{eqnarray*}
\Pr{\log\Gamma_{0,\gamma} e^{Z_\xi(\theta)} > \log(1/\alpha)} &\leq& \alpha\,\Ex{\Gamma_{0,\gamma}e^{S_\xi(\theta)}}\\
&=&\alpha\, \Gamma_{0,\gamma}\Ex{e^{S_\xi(\theta)}}.
\end{eqnarray*}
Furthermore, a computation with moment generating functions as in the proof of Bernstein's  inequality \cite[2.8]{boucheron2013concentration} gives 
\[\forall\theta\in\R^d: \Ex{e^{S_\xi(\theta)}}=\prod_{i=1}^n\left(\Ex{\exp\left\{\ \xi(Z_i(\theta)-\Ex{Z_i(\theta)})- \frac{\xi^2\Var{Z_i(\theta)}}{2-\frac{2\xi A}{3}} \right\} } \right) \leq 1.
\]
Combining these two statements we obtain (\ref{eq:claim1PAC}). 

Next, we claim that for a specific choice of  $\xi^*\in(0, 3/A)$, if $\log\Gamma_{0,\gamma}e^{S_{\xi^*}(\theta)}\leq \log(1/\alpha)$, then 
\begin{equation}\label{eq:claim2PAC} \sup_{v\in\Sd}\Gv\left(\sum_{i=1}^n Z_i(\theta)\right) \leq n\bar{\mu}_\gamma+ \bar{\sigma}_\gamma\sqrt{n(\gamma^{-2}+2\,\log(1/\alpha))} +\frac{A \left(\gamma^{-2}+2\,\log(1/\alpha)\right)}{6}.\end{equation}
This implies that 
\begin{eqnarray*}
\Biggl\{\sup_{v\in\Sd}\Gv\left(\sum_{i=1}^n Z_i(\theta)\right) \leq  n\bar{\mu}_\gamma+ \bar{\sigma}_\gamma\sqrt{n(\gamma^{-2}+2\,\log(1/\alpha))} +\frac{A \left(\gamma^{-2}+2\,\log(1/\alpha)\right)}{6}\Biggl\}\\\ \supset \{\log\Gamma_{0,\gamma}e^{S_{\xi^*}(\theta)}\leq  \log(1/\alpha)\}\end{eqnarray*}
whereas (\ref{eq:claim1PAC}) lower bounds the probability of the smaller event by $1-\alpha$. In particular, (\ref{eq:claim1PAC}) and (\ref{eq:claim2PAC}) together imply the present proposition. 

To prove (\ref{eq:claim2PAC}), we first note that the 
definitions of $\bar{\mu}_\gamma$ and $\bar{\sigma}_\gamma^2$ imply the following:
\[\frac{\Gamma_{v,\gamma}\,S_\xi(\theta)}{\xi}\geq \sum_{i=1}^n\,\Gamma_{v,\gamma}\,Z_i(\theta) - n\bar{\mu}_\gamma -\frac{n\,\xi\,\bar{\sigma}^2_\gamma }{2\left(1-\xi A/3 \right)}.\]
Hence, going back to (\ref{eq:EIconseqGauss}), we have that if $\Gamma_{0,\gamma}\Ex{e^{S_\xi(\theta)}}\leq \log(1/\alpha)$, 
\[\sup_{v\in\Sd}\sum_{i=1}^n\Gv\left( Z_i(\theta)\right) \leq  n\bar{\mu}_\gamma+ \frac{n\xi\bar{\sigma}^2_\gamma}{2-\frac{2A\xi}{3}} +\frac{\frac{\gamma^{-2}}{2}+\log{(1/\alpha)}}{\xi}.\] The last display holds for any $\xi\in(0, 3/A)$. Choosing
\[\xi^*:=\frac{\sqrt{\gamma^{-2}+2\log(1/\alpha)}}{\sqrt{n}\,\bar{\sigma}_\gamma\left(1+\frac{A\sqrt{\gamma^{-2}+2\log(1/\alpha)}}{3\sqrt{n}\,\bar{\sigma}_\gamma} \right)}\]
it follows that
\[\frac{n\xi^*\ \bar{\sigma}^2_\gamma}{2-\frac{2A\xi^*}{3}} +\frac{\frac{\gamma^{-2}}{2}+\log{(1/\alpha)}}{\xi^*} = \bar{\sigma}_\gamma\sqrt{n(\gamma^{-2}+2\log(1/\alpha))}+\frac{A(\gamma^{-2}+2\log(1/\alpha))}{6}.\]
\end{proof} 

\section{The covariance estimator, and a warm-up exercise}\label{sec:overview}

In this section, we define trimmed mean estimators of $\Sigma$ with different trimming parameters (\S \ref{sub:defestimator}). We discuss in \S \ref{sub:simple} how trimmed means can be used to estimate means of scalar random variables: this is a warm-up exercise and gives a useful result for estimating $\tr{\Sigma}$. 

\subsection{Basics}\label{sub:defestimator}

Make Assumption \ref{assum1}. For $k\in\Na$, $k<n$, and each unit vector $v\in \Sd$, define:
\begin{equation}\label{eeq:defestk}\est_k(v):=\inf\limits_{S\subset [n],\,\# S = n-k }\frac{1}{n-k}\sum_{i\in S}\inner{Y_i}{v}^2,\end{equation}
where the $Y_i$ form an $\eta$-corruption of the $X_i$. These quantities will be our estimates for the values $\inner{v}{\Sigma v}$; different choices of $k$ correspond to different bias-variance trade-offs.  

The next proposition shows that, if $\,\est_k(v)$ estimates $\inner{v}{\Sigma v}$ well for all unit vectors $v\in \Sd$, then one can find a good estimator for the covariance matrix itself.

\begin{proposition}\label{prop:estimatorisdefined}There exists a random element $\Est_k$ of $\R^{d\times d}_{\geq 0}$ such that:
\[\Est_k\in {\rm arg}\min_{A\in \R^{d\times d}_{\geq 0}}\left(\sup_{v\in\Sd}|\inner{v}{Av} -\est_k(v)|\right).\]
Moreover, for any such estimator, $\|\Est_k - \Sigma\|\leq 2\sup_{v\in \Sd}|\inner{v}{Av} -\est_k(v)|.$
\end{proposition}

This proposition will give us a family of estimators, one for each trimming parameter $k$. Our final estimator will require a data-driven choice of the parameter $k$.

\begin{proof}[Proof of Proposition \ref{prop:estimatorisdefined}] Notice that the functions mapping $Y_1,\dots,Y_n,A$ to $|\inner{v}{Av} -\est_k(v)|$ are all Lipschitz-continuous, and also divergent when $\|A\|\to +\infty$. It follows that:
\[\inf_{A\in \R^{d\times d}_{\geq 0}}\left(\sup_{v\in\Sd}|\inner{v}{Av} -\est_k(v)|\right)\]
is well-defined, measurable (the sup and inf can be taken over countable sets), and achieved by at least one $A$. Finding a measurable $\Est_k$ is then a simple application of the Kuratowski-Ryll-Nardzewski Measurable Selection theorem.

For the inequality, let $H_k(A):=\left(\sup_{v\in\Sd}|\inner{v}{Av} -\est_k(v)|\right)$. Then
\[\|\Est_k - \Sigma\| = \sup_{v\in \Sd}|\inner{v}{\Est_k\,v} -\inner{v}{\Sigma v}|\leq  H_k(\Est_k)+H_k(\Sigma),\]
by the triangle inequality. Additionally, $H_k(\Sigma)\geq H_k(\Est_k)$ because $\Est_k$ minimizes $H_k$.\end{proof}

\begin{remark}[Computational aspects]\label{rem:computational} The estimators $\Est_k$ are clearly not computable. A way around this would be to redefine 
\[\Est_k\in {\rm arg}\min_{A\in \R^{d\times d}_{\geq 0}}\left(\sup_{v\in\N}|\est_k(v)-\inner{v}{Av}|\right),\]
where $\N$ is a $(1/4)$-net of the unit sphere. This would be good enough for our purposes because it is well-known that:
\[\|\Est_k - \Sigma\|\leq 2\sup_{v\in \N}|\inner{v}{(\Est_k - \Sigma)v}|.\]
This new estimator can be computed as the minimum of a convex function. However, since the size of $\N$ is exponential in the ambient dimension, this the method thus obtained is not particularly efficient. This is why we choose to work with the cleaner version of $\Est_k$ defined above. \end{remark}

\subsection{A warm-up: trimming via  truncation}\label{sub:simple}

As noted in \S \ref{sub:approach}, we analyse our trimmed mean estimators $\est_k(v)$ via truncation and counting conditions. To illustrate these points, we briefly consider the simpler problem of estimating the mean of a nonnegative random variable under contamination. This will highlight how trimming and truncation are related, and will be needed for estimating $\tr{\Sigma}$ later.

\subsubsection{Estimating the mean of a nonnegative random variable}\label{sub:nonnegative}

We temporarily switch to the following assumption. 
\begin{assumption}$Z_1,\dots,Z_n$ be i.i.d. nonnegative real variables with $\Ex{Z_1^q}<+\infty$ for some $q\geq 2$. $W_1,\dots,W_n$ be other random variables that satisfy:
\[\#\{i\in [n]\,:\,Z_i\neq W_i\}\leq \lfloor \eta n\rfloor\mbox{ for some }\eta\in [0,1).\]\end{assumption}

The mean of $Z_1$ will be estimated from the $W$-sample via:
\begin{equation}\label{eq:defTk}T_k:=\inf\limits_{S\subset [n],\# S=k}\sum_{i\in S}\frac{W_i}{n-k},\end{equation}
for a suitable value of $k$. Equivalently, $T_k$ is defined by taking the average of all but the $k$ largest values amongst the $W_i$. We will prove the following result. 

\begin{proposition}\label{prop:trimandtruncate} Under the above assumption, fix $\alpha\in (0,1)$ and $c>0$. Assume:
\[k:=\lfloor \eta n\rfloor + \lceil c\,\eta n +\log(2/\alpha)\rceil<n.\] Then
\[\Pr{|T_k - \Ex{Z_1}|\leq C\sqrt{\frac{\Ex{Z_1^2}\,\log(2/\alpha)}{n}} + C_c\Ex{Z_1^q}^\frac{1}{q}\,\eta^{1-\frac{1}{q}}}\geq 1-\alpha,\]
where $C_c>0$ depends only on $c>0$.\end{proposition}

This result is similar to the first sections of \cite{lugosi2021}, where they credit Orenstein and Oliveira for this approach (the actual end line of this research direction is chapter 2 of Rico's thesis \cite{rico2022}). 

We also note that we are using asymmetric trimming in $T_k$, where only the largest values in the sample may be removed. This is convenient for nonnegative random variables, but leads to bounds in terms of uncentered moments of $Z_1$. See \cite{rico2022} and \cite{oliveira2023trimmed} for examples of symmetrical trimming, which leads to centered bounds.

What is interesting about Proposition \ref{prop:trimandtruncate} is that, even though $T_k$ is computed on the $W_i$, and even though the $Z_i$ may have heavy tails, $T_k$ achieves an error that is the sum of a ``sub-Gaussian term''~(proportional to $\sqrt{\log(2/\alpha)/n}$) and a term coming from contamination (proportional to $\eta^{1-1/q}$). Neither term can be significantly improved \cite{dllo2016,minsker2018sub}.

To prove Proposition \ref{prop:trimandtruncate}, we follow a strategy with three steps. The first one is a purely deterministic claim: under a suitable {\em counting condition}, the trimmed mean $T_k$ is closed to a truncated mean of the uncontaminated sample. This kind of reasoning will also be useful to obtain our main results. 

\begin{claim}\label{claim:trimandtruncate}Assume $b>0$ and $t\leq k-\lfloor \eta n\rfloor$. If the following event holds:
\[{\rm count}(b,t):=\{\#\{i\in [n]\,:\,Z_i>b\}\leq t\},\]
we have:
\[\left|T_k - \frac{1}{n}\sum_{i=1}^nZ_i\wedge b\right|\leq \frac{2bk}{n} \mbox{ and }|\Ex{Z_1} - \Ex{Z_1\wedge B}|\leq \frac{\Ex{Z_1^q}}{b^{q-1}}.\]\end{claim}

\begin{proof}The part about expectations follows from the deterministic bound:
\[|Z_1 - Z_1\wedge b| = (Z_1-b)_+\leq \frac{Z_1^q}{b^{q-1}}.\]
For the other inequality, let $S^*\subset [n]$ denote the set of indices $i\in [n]$ achieving the $n-k$ smallest values of $W_i$ (if there is more than one such set, choose one arbitrarily). Then:
\[T_k = \frac{1}{n-k}\sum_{i\in S^*}W_i.\]
Under the ${\rm count}(b,t)$ event, it must be that $W_i\leq b$ for all $i\in S^*$. After all, ${\rm count}(b,t)$ implies that there can only be at most $t$ indices $j$ with $Z_j>b$. Therefore, by the contamination condition, at most $t+\lfloor \eta n\rfloor \leq k$ indices $j$ will satisfy $W_j>b$, and none of these indices will lie in $S^*$. We deduce
\[T_k = \frac{1}{n-k}\sum_{i\in S^*}(W_i\wedge b).\]
Now notice that the indices $i\in S^*$ also correspond to the $n-k$ smallest values of $W_i\wedge b$. Averaging over all indices can only give a larger result, so:
\[T_k = \frac{1}{n-k}\sum_{i\in S^*}(W_i\wedge b)\leq \frac{1}{n}\sum_{i=1}^n(W_i\wedge b).\]
On the other hand, since $\# S^*=n-k$,
\[T_k \geq \frac{1}{n}\sum_{i\in S^*}(W_i\wedge b)\geq \frac{1}{n}\sum_{i=1}^n(W_i\wedge b) - \frac{bk}{n}.\]
Thus:
\[\left|T_k - \frac{1}{n}\sum_{i=1}^n(W_i\wedge b) \right|\leq  \frac{bk}{n}.\]
But also, 
\[\left|\frac{1}{n}\sum_{i=1}^n(Z_i\wedge b) - \frac{1}{n}\sum_{i=1}^n(W_i\wedge b) \right|\leq \frac{bk}{n},\]
because all terms in the two averages lie between $0$ and $b$, and at most $\lfloor \eta n\rfloor \leq k$ of them are different.\end{proof}

The next question becomes how one finds a $b$ ensuring that ${\rm count}(b,t)$ holds with high probability. This is the content of the next claim.

\begin{claim}\label{claim:count}For any integer $t>0$ we have $\Pr{{\rm count}(b,t)}\geq 1-e^{-t-2}$ for
\[b=b_q(t):=e^{\frac{2}{q}}\left(\frac{n\,\Ex{Z_1^q}}{t}\right)^{\frac{1}{q}}.\]\end{claim}

\begin{proof}The event ${\rm count}(b,t)$ does {\em not} hold if and only if there exists a set of $t+1$ distinct indices $i\in [n]$ with $Z_i>b$. By a union bound, and using the fact that the $Z_i$ are i.i.d., we obtain:
\[1-\Pr{{\rm count}(b_q(t),t)}\leq\binom{n}{t+1}\left(\frac{\Ex{Z_1^q}}{b_q(t)^q}\right)^{t+1}\leq \left(\frac{en\,\Ex{Z_1^q}}{(t+1)\,b_q(t)^q}\right)^{t+1}\leq e^{-t-2}.\]
\end{proof}

Finally, we now conclude the proof of Proposition \ref{prop:trimandtruncate}.

\begin{proof}[Proof of Proposition \ref{prop:trimandtruncate}] 
Choose $t=\lceil c\,\eta n +\log(2/\alpha)\rceil$ and $b=b_q(t)$ as in Claim \ref{claim:count}. Combining that claim with Bernstein's inequality, we see that both ${\rm count}(b,t)$ and the inequality
\[\left|\frac{1}{n}\sum_{i=1}^nZ_i\wedge b- \Ex{Z_1\wedge b} \right|\leq C\sqrt{\frac{\Ex{Z_1^2}\log(2/\alpha)}{n}} + C\,\frac{b\log(2/\alpha)}{n}\]
hold simultaneously with probability $\geq 1-\alpha$, for some universal $C>0$. 

Let us now assume that these two events above hold. If we further employ Claim \ref{claim:trimandtruncate}, we can deduce the following:
\[|T_k-\Ex{Z_1}|\leq  C\sqrt{\frac{\Ex{Z_1^2}\log(2/\alpha)}{n}} + C\,\frac{b\log(2/\alpha)}{n} + \frac{2bk}{n}+\frac{\Ex{Z_1}^q}{b^{q-1}
}.\]
Now recall that $b=(e^2\Ex{Z_1}^qn/t)^{1/q}$, $t = \lceil c\eta n + \log(2/\alpha)\rceil$, and notice that $k\leq (c+1)t/c$. This gives us the bound:
\[\Pr{|T - \Ex{Z_1}|\leq C\sqrt{\frac{\Ex{Z_1^2}\,\log(2/\alpha)}{n}} + C_c\Ex{Z_1^q}^\frac{1}{q}\,\left(\eta + \frac{\log(2/\alpha)}{n}\right)^{1-\frac{1}{q}}}\]
for some $C_c>0$ depending only on $c$. This gives us the inequality in the Proposition (with a different $C_c$) if $\eta>\log(2/\alpha)/n$; otherwise, we may repeat the proof with $q=2$ (which is allowed) and obtain:
\[\Pr{|T - \Ex{Z_1}|\leq C_c\sqrt{\frac{\Ex{Z_1^2}\,\log(2/\alpha)}{n}}}\]
and obtain the inequality in the Proposition (again, up to the specific value of $C_c$).\end{proof}

\subsubsection{Estimating the trace}\label{sub:estimatetrace}

We now note in passing that the methods of \S \ref{sub:nonnegative} allow us to estimate $\tr{\Sigma}$ up to a small multiplicative factor. This simple observation will be extremely useful in \S \ref{sub:finalfinal}, where we finish the proof of the main result. Assumption \ref{assum1} is enforced throughout this subsection. 

\begin{corollary}\label{cor:trace}Under Assumption \ref{assum1}, fix $\alpha\in (0,1)$ and $c>0$ and assume \[k:=\lfloor \eta n\rfloor + \lceil c\,\eta n +\log(2/\alpha)\rceil<n.\]
Define the following estimator, 
\[\Trest:=\inf_{S\subset [n],\# S=n-k}\frac{\sum_{i\in S}\|Y_i\|^2}{n-k}.\]
Then, for some $C_c>0$ depending only on $c$,
\[\Pr{\frac{|\Trest - \tr{\Sigma}|}{\tr{\Sigma}}\leq C_c\kappa_4^2\sqrt{\frac{\log(2/\alpha)}{n}} + C_c\kappa_p^2\,\eta^{1-\frac{2}{p}}}\geq 1-\alpha.\]
\end{corollary}

Indeed, the corollary is a consequence of Proposition \ref{prop:trimandtruncate} applied with $Z_i:=\|X_i\|^2$, $W_i:=\|Y_i\|^2$ and $q:=p/2$, combined with the next proposition (which will be useful elsewhere in the paper). 

\begin{proposition}\label{prop:Minkowski} For all $p\geq 2$, $(\Ex{\|X_1\|^p})^{1/p}\leq \kappa_p\sqrt{\tr{\Sigma}}$.\end{proposition}
\begin{proof}Let $\{e_i\,:\,1\leq i\leq d\}$ denote an orthonormal basis of eigenvectors for $\Sigma$, with respective eigenvalues $\lambda_i$. For each $i\in[d]$, $\inner{e_i}{\Sigma e_i} = \Ex{\inner{X_1}{e_i}^2} = \lambda_i$, so our assumptions imply $\Ex{\inner{X_1}{e_i}^p}\leq (\kappa^2_p\lambda_i)^{p/2}.$
Now:
\[\|X_1\|^2 = \sum_{i=1}^d\inner{e_i}{X_1}^2,\]
and Minkowski's inequality for the $L^{p/2}$ norm gives:
\[\Ex{\|X_1\|^p}^{\frac{2}{p}} = \Ex{\left(\sum_{i=1}^d\inner{e_i}{X_1}^2\right)^{\frac{p}{2}}}^{\frac{2}{p}}\leq \sum_{i=1}^d\,\Ex{\inner{X_1}{e_i}^p}^{\frac{2}{p}}\leq  \kappa^2_p\sum_{i=1}^d\lambda_i =\kappa_p^2\tr{\Sigma}.\]
\end{proof}

\section{Trimming the covariance and counting over the unit sphere}\label{sec:countingforvectors}

We now go back to estimating covariances under Assumption \ref{assum1}. In \S \ref{sub:defestimator} we argued that estimating $\Sigma$ could be reduced to the problem of estimating $\langle v,\Sigma v\rangle$ uniformly in $v\in \Sd$ via the trimmed mean estimators $\est_k(v)$. As a first step, we relate the error in these estimated to the truncated empirical process
\begin{equation}\label{eq:definitionvarepsilon}\varepsilon(B):=\sup_{v\in\Sd} \frac{1}{n}\abs{\sum_{i=1}^n \left( \inner{X_i}{v}^2 \land B-\Ex{ \inner{X_i}{v}^2 \land B}\right) },\end{equation}
where $B>0$ is a constant. This will require the introduction of certain uniform counting events. For $B,t>0$, define:
\begin{equation}\label{eq:defcount}{\rm Count}(B,t):=\left\{\forall v\in\Sd: \#\{i\in [n] \,:\,\inner{X_i}{v}^2>B\} \leq t\right\}.\end{equation}

\begin{proposition}\label{prop:estimatortruncatedempirical}Assume $B>0$, $t\in \Na$ and $k\geq \lfloor \eta n \rfloor + t$. Then, whenever the event ${\rm Count}(B,t)$ holds, 
\[\sup_{v\in \Sd}|\est_k(v) - \inner{v}{\Sigma v}|\leq \varepsilon(B) + \frac{2Bk}{n} +  \frac{\kappa_p^p\|\Sigma\|^{\frac{p}{2}}}{B^{p-2}}.\]\end{proposition}

\begin{proof}We show the stronger statement that for any $v\in\Sd$,
\begin{equation}\label{eq:strongerestimatortruncatedempirical}|\est_k(v) - \inner{v}{\Sigma v}|\leq \frac{1}{n}\abs{\sum_{i=1}^n \left( \inner{X_i}{v}^2 \land B-\Ex{ \inner{X_i}{v}^2 \land B}\right) }+ \frac{2Bk}{n} +  \frac{\kappa_p^p\|\Sigma\|^{\frac{p}{2}}}{B^{p-2}}.\end{equation}
To see this, fix some $v\in \Sd$. Recalling the notation in \S \ref{sub:nonnegative}, we set $Z_i:=\inner{X_i}{v}^2$ and $W_i:=\inner{Y_i}{v}^2$ for each $i=1,2,\dots,n$. Then the estimator $\est_k(v)$ in (\ref{eeq:defestk}) corresponds to $T_k$ in (\ref{eq:defTk}), and $\Ex{Z_1}=\langle v,\Sigma,v\rangle$. 

To finish the proof of (\ref{eq:strongerestimatortruncatedempirical}), we set $b:=B$ and apply Claim \ref{claim:count}, checking that event ${\rm Count}(B,t)$ contains the event ${\rm count}(b,t)$ in the statement of the Claim, and also that $\Ex{|Z_1|^p}\leq \kappa_p^p\sqrt{\langle v,\Sigma v\rangle}\leq \kappa_p^p\|{\Sigma}\|$.\end{proof}

We are now left with the tasks of estimating $\Pr{{\rm Count}(B,t)}$ and the magnitude of $\varepsilon(B)$. This turns out to be much harder than the reasoning in \S \ref{sub:simple}. The remainder of this section takes care of the counting event via the following Lemma. 

 \begin{lemma}[Counting Lemma] \label{lem:counting} Under Assumption \ref{assum1}, pick $t\in \Na$ and set:
\[B_p(t) := \|\Sigma\|\,\left[\left(4\,\kappa_p^2\,\left(\frac{2000\,n}{t}\right)^{\frac{2}{p}}\right)\vee \left(512\,\kappa_4^2\,\rs\,\frac{\sqrt{n}}{t^{3/2}}\right)\right].\]
Then:\begin{equation*}
\Pr{{\rm Count}(B_p(t),t)}\geq 1-e^{-t}. \end{equation*}
 \end{lemma}

To gain some intuition for the Lemma \ref{lem:counting}, note that, when $p=4$,
\begin{equation}\label{eq:simplifyBt}B_p(t)\leq C\kappa_4^2\|\Sigma\|\sqrt{\frac{n}{t}}\left(1 + \frac{\rs}{t}\right), \end{equation}
for some universal $C$. Except for the $\rs/t$ term, this is basically the value of $b_q(t)$ that Claim \ref{claim:count} gives for a single choice of $v\in \Sd$ (taking $Z_i:=\inner{X_i}{v}^2$ and $q=2$). 

What Lemma \ref{lem:counting} is saying, with the simplified bound for $B_p(t)$ in (\ref{eq:simplifyBt}), is that the single-vector counting bound becomes a uniform bound over $v\in\Sd$ as soon as $t\geq \rs$. That is, for large enough $t$ there is essentially no price to pay to obtain a uniform bound. Incidentally, this explains why our estimator will require trimming more than $\rs$ terms (up to constants). 

\begin{remark}[Other counting results]\label{rem:othercountingresults} Similar ``counting results'' have appeared in previous work \cite{lugosi2021,oliveira2023trimmed}. The closest example to Lemma \ref{lem:counting} is by Abdalla and Zhivotovskiy \cite[Lemma 6]{abdalla2022}. Their proof obtains the counting condition from a bound on the Rademacher complexity of a quadratic empirical process. This requires distinguishing between the $p=4$ and $p>4$ cases, as the Bai-Yin-type bounds for these process are different. By contrast, our proof of Lemma \ref{lem:counting} is a direct argument via PAC-Bayesian methods that does not require previous control of the empirical process. Later in the proof, this will allow us to deal only with a truncation of the quadratic process, in contrast with \cite[Section 2]{abdalla2022}.\end{remark}

%\begin{remark}The term containing $\rs$ in $B_p(t)$ could be improved for $p>4$, but such improvements would not matter for our proof.\end{remark}

The proof of Lemma \ref{lem:counting} requires a preliminary calculation that we present separately as the next proposition.

\begin{proposition}\label{prop:choiceofBcounting}The choice of $B_p(t)$ in the Counting Lemma ensures that, for $\gamma = \sqrt{2/t}$:
\[\sup_{v\in \Sd}\Gamma_{v,\gamma}\Pr{\inner{X_1}{\theta}^2>B_p(t)}\leq \frac{t}{1000n}.\]\end{proposition}
\begin{proof}[Proof of Proposition \ref{prop:choiceofBcounting}]For all $v\in \Sd$, $B>0$:
\begin{eqnarray*}\Gamma_{v,\gamma}\Pr{\inner{X_1}{\theta}^2>B}&\leq & \Pr{|\inner{X_1}{v}|\geq \frac{\sqrt{B}}{2}} + \Gamma_{v,\gamma}\Pr{|\inner{X_1}{\theta-v}|>\frac{\sqrt{B}}{2}}\\ &\leq & \frac{2^p\Ex{|\inner{X_1}{v}|^p}}{B^{p/2}} + \frac{16\Gamma_{v,\gamma}\Ex{\inner{X_1}{\theta-v}^4}}{B^2}
\\
(v\in\Sd)&\leq & \left(\frac{4\kappa^2_p\|\Sigma\|}{B}\right)^{\frac{p}{2}} + \frac{16\gamma^4\Gamma_{0,1}\Ex{\inner{X_1}{\theta}^4}}{B^2}. \end{eqnarray*}
The first term in the RHS is $\leq t/2000n$ for our choice $B=B_p(t)$. We now bound the second term, thus finishing the proof. Indeed, 
\[\Ex{\inner{X_1}{\theta}^4}\leq \kappa_4^4 \inner{\theta}{\Sigma \theta}^2.\]
When $\theta$ has law $\Gamma_{0,1}$, one can use rotational invariance to deduce that  $\inner{\theta}{\Sigma\theta}$ has the same distribution as $\sum_{i=1}^d\lambda_iN_i^2$, where the $N_i$ are i.i.d. standard normals and the $\lambda_i$ are the eigenvalues of $\Sigma$. Using this, one sees that:
\[\Gamma_{0,1}\inner{\theta}{\Sigma \theta}^2 = \sum_{i=1}^d 3\lambda_i^2 + 2\sum_{1\leq i<j\leq d}\lambda_i\lambda_j\leq 3\tr{\Sigma}^2.\]
Recalling that $\gamma^4=4/t^2$, we obtain:
\[\frac{16\gamma^4\Gamma_{0,1}\Ex{\inner{X_1}{\theta}^4}}{B^2}\leq \frac{2^6\kappa_4^4\tr{\Sigma}^2}{t^2B^2} .\]
Our choice of $B=B_p(t)$ guarantees this term is at most:
\[\frac{1}{2^{12}}\frac{t}{n}\leq \frac{t}{2000n},\]
as claimed. \end{proof}

 \begin{proof}[Proof of Lemma \ref{lem:counting}]Write $B:=B_p(t)$ for simplicity, and also set $\gamma=\sqrt{2/t}$ as in Proposition \ref{prop:choiceofBcounting}. A key observation for the proof is the following: for all $v\in\Sd$,
\begin{equation} \label{rel:counting} \forall i\in [n]\,:\,\frac{\Ind{\inner{X_i}{v}^2\geq B}}{2} \leq \Gv \Ind{\inner{X_i}{\theta}^2\geq B}.\end{equation} 

To see this, notice that, when $\theta$ is distributed according to the Gaussian measure $\Gamma_{v,\gamma}$, the distribution of $\theta-v$ is symmetric around $0$. Therefore, if $\inner{X_i}{v}^2\geq B$, there is probability $1/2$ of $\inner{X_i}{\theta-v}$ having the same sign as $\inner{X_i}{v}$, in which case $\inner{X_i}{\theta}^2\geq \inner{X_i}{v}^2\geq B$ also.

The main consequence of  (\ref{rel:counting}) is that the event ${\rm Count}(B,t)$ satisfies:
\[{\rm Count}(B,t)\supset \left\{\sup_{v\in \Sd}\sum_{i=1}^n \Gamma_{v,\gamma}\Ind {\inner{X_i}{\theta}^2 > B}\leq \frac{t}{2}\right\}.\]Therefore, the Lemma will follow from the following bound:
\begin{equation}\label{eq:goalcountinglemma} \Pr{\sup_{v\in \Sd}\sum_{i=1}^n \Gamma_{v,\gamma}\Ind {\inner{X_i}{\theta}^2 > B}\leq \frac{t}{2}}\geq 1-e^{-t}.\end{equation} 

To obtain this goal, we apply Proposition \ref{th:PacBernstein} to the random variables $Z_i(\theta):=\Ind{\inner{X_i}{\theta}^2> B}$ with $\gamma=\sqrt{2/t}$. To continue, we must compute the values of $\overline{\mu}_\gamma$ and $\overline{\sigma}_\gamma$. In the present setting, the random variables $Z_i$ are indicators, so $\overline{\sigma}^2_\gamma\leq \overline{\mu}_\gamma$. Moreover,   
\[\overline{\mu}_\gamma = \sup_{v\in \Sd}\Gamma_{v,\gamma}\Ex{Z_i(\theta)} \leq \frac{t}{1000\,n},\]
We obtain from Proposition \ref{th:PacBernstein} that, with probability $\geq 1-e^{-t}$
\[ \sup_{v\in \Sd}\sum_{i=1}^n\Gamma_{v,\gamma} \Ind {\inner{X_i}{\theta}^2 >B}\leq
\frac{t}{1000} + \sqrt{\frac{t}{1000}\,(2t+\gamma^{-2})}
+\frac{2\,t+\gamma^{-2}}{6} \leq \frac{t}{2},
\]
with some room to spare in the constants (recall $\gamma^{-2}=t/2)$. This gives us (\ref{eq:goalcountinglemma}).\end{proof}

\section{The truncated empirical processes}\label{sec:vectors}
The previous section controls the counting event needed to bound the error in our estimator. By Proposition \ref{prop:estimatortruncatedempirical} above, this allows us to relate the error of the covariance estimator to the truncated empirical process in (\ref{eq:definitionvarepsilon}). 

The main result of this section is a concentration bound for the truncated process, with suitable $B>0$. 

\begin{lemma}[Proof in \S \ref{sub:proof:lem:truncatedconcentrates}]\label{lem:truncatedconcentrates} Make Assumption \ref{assum1} Take $\alpha\in (0,1)$ and $t\in\Na$ with $t\geq 2$. Choose $B_p(t)$ as in the Counting Lemma (Lemma \ref{lem:counting}). Then the following holds with probability $\geq 1-\alpha$
\begin{equation}\label{eq:lemma}\varepsilon(B_p(t))\leq  C\,\kappa_4^2\,\|\Sigma\|\sqrt{\frac{t + 2\log(8/\alpha)}{n}} + C\,\kappa_p^2\,\|\Sigma\|\left(\frac{t}{n}\right)^{1-\frac{2}{p}}+ C\,\kappa_4^2\,\|\Sigma\| \frac{\rs}{\sqrt{nt}},\end{equation}
where $C>0$ is a universal constant independent of all other problem parameters.\end{lemma}

The proof of this lemma will require a few steps. As a first step, we will introduce a Gaussian-smoothed version of the empirical process,
\begin{equation}\label{eq:defsmoothedtruncated}\widetilde{\varepsilon}_\gamma(B'):=\sup_{{v\in\Sd}} \frac{1}{n}\abs{\sum_{i=1}^n\Gv \left( \inner{X_i}{\theta}^2 \land B'-\Ex{ \inner{X_i}{\theta}^2\land B'}\right)},\end{equation}
with a suitable choice of $\gamma>0$ and a parameter $B'$ potentially different from $B$. We show in \S \ref{sub:smoothedlemma} that the PAC-Bayesian Bernstein inequality (Theorem \ref{th:PacBernstein}) readily implies concentration for this process; see Lemma \ref{th:boundSmooth} for a precise statement.

In a second step, we will compare the terms in the sums defining the smoothed and unsmoothed empirical process. This is the content of the fairly technical Lemmas \ref{lem:comparisonsmoothed} and \ref{lemma:NormBounds}, stated in \S \ref{sub:smoothedlemma} and proven in the Appendix. 

Finally, \S \ref{sub:proof:lem:truncatedconcentrates} proves Lemma \ref{lem:truncatedconcentrates} via a combination of the above ingredients, the Counting Lemma and another (simpler) counting result for the norms $\|X_i\|$.

\subsection{The smoothed empirical process}\label{sub:smoothedlemma}

We give here a concentration bound for the smoothed empirical process. 

 \begin{lemma} \label{th:boundSmooth}
Make Assumption \ref{assum1}. Consider $B'>0$ and $\alpha_0\in (0,1)$. If $\gamma^{-2}:=t$, then:
\[
\widetilde{\varepsilon}_\gamma(B')
\leq 2\kappa_4^2\,\|\Sigma\|\sqrt{1+\frac{3\rs^2}{t^2}}\sqrt{\frac{t+2\log(2/\alpha_0)}{n}
}  
+ \frac{B'}{6\,n}(t+2\log(2/\alpha_0)),
\]
 with probability at least $1-\alpha_0$.
 \end{lemma} 
 
\begin{proof}
We apply Theorem \ref{th:PacBernstein} to the upper and lower tails of the empirical process. In the notation of that theorem, we set $\gamma^{-2}:=t$, take $\alpha:=\alpha_0/2$ and $A=B'$, and set:
\begin{eqnarray*}
Z_i(\theta)&=& \pm\frac{1}{n} (\insq{X_i}{\theta}\land B' - \Ex{\insq{X_i}{\theta}\land B'})_{i\in[n], \theta\in\R^d}, \\ \overline{\mu}_\gamma &=&0, \text{ and }  \\ \bar{\sigma}^2_\gamma&=& S^2 :=\sup_{v\in\Sd}\sum_{i=1}^n\Gv  \Var{\insq{X_i}{\theta}\land B'}.
\end{eqnarray*}
The following inequality holds for all $v\in\Sd$ with probability at least $1-\alpha$,
\begin{eqnarray*} \label{eq:sumB_xi}
\sum_{i=1}^n\Gv \left( \insq{X_i}{\theta} \land B'-\Ex{ \insq{X_i}{\theta} \land B'}\right) &\leq & S\,\sqrt{(t+2\log(1/\alpha))} \\ &  +& \frac{(t+2\log(1/\alpha))B'}{6}.
\end{eqnarray*}

To finish the proof, we show the bound \[S\leq 2\kappa_4^2\,\|\Sigma\|\sqrt{1+\frac{3\rs^2}{t^2}}.\] Indeed, for all $v\in\Sd$,\begin{eqnarray*}
\Gv\Var{\insq{X_i}{\theta}  \land B'} &=& \Gv\Ex{\inq{X_i}{\theta}} - \Gv\left(\Ex{\insq{X_i}{\theta}} \right)^2\\
&\leq& \Gamma_{0,1}\Ex{\inq{X_i}{v+\gamma\,\theta}}\\
& = & \Gamma_{0,1} \Ex{\inner{X_i}{v}^4 + \gamma^4\inner{X_i}{\theta}^4 + 6\gamma^2\inner{X_i}{v}^2\,\inner{X_i}{\theta}^2}\\
&\leq & \Gamma_{0,1} \Ex{4\inner{X_i}{v}^4 + 4 \gamma^4\inner{X_i}{\theta}^4}\\ & \leq & 4\kappa_4^4\,\Gamma_{0,1}(\inner{v}{\Sigma v}^2 + \gamma^4\inner{\theta}{\Sigma \theta}^2)\\ &\leq &  4\kappa_4^4\,(\|\Sigma\|^2 + 3\gamma^4\tr{\Sigma}^2) = 4\kappa_4^4\,\left(\|\Sigma\|^2 + 3\frac{\tr{\Sigma}^2}{t^2}\right), %\\ &\leq & 16\kappa_4^4\,\|\Sigma\|^2\,
\end{eqnarray*}
%where the last inequality uses $\gamma^2\rs\leq 1$, and the second-to-last inequality uses $\Gamma_{0,1}\inner{\theta}{\Sigma \theta}^2\leq 3\tr{\Sigma}^2$ (cf. the end of the proof of Proposition \ref{prop:choiceofBcounting}).\end{proof}
 where the last inequality uses $\Gamma_{0,1}\inner{\theta}{\Sigma \theta}^2\leq 3\tr{\Sigma}^2$ (cf. the end of the proof of Proposition \ref{prop:choiceofBcounting}).\end{proof}

In what follows, we will now need to bound the difference between the smoothed and unsmoothed truncated empirical processes $\widetilde{\varepsilon}_\gamma(B')$ and ${\varepsilon}(B)$  defined above. This step and other parts of the proof will require control of the smoothing operation. The next two lemmas contain the results we will need. 

\begin{lemma}[Proof in Appendix \S \ref{appendixA}] \label{lem:comparisonsmoothed} For any $B>0$ and $\gamma>0$,
\begin{eqnarray*}\varepsilon(B)- \widetilde{\varepsilon}_\gamma(4B)&\leq& \left|\frac{1}{n}\sum_{i=1}^n[(\gamma\|X_i\|)^2\wedge B - \Ex{(\gamma\|X_1\|)^2\wedge B }]\right| \\  & & + \frac{4B}{n}\sup_{v\in\Sd}\#\{i\in[n]\,:\,\inner{X_i}{v}^2>B\}\\ & & + \frac{4B}{n}\#\{i\in[n]\,:\,(\gamma\|X_i\|)^2>B\} \\ & & + \frac{1}{n}\sum_{i\in[n]\,:\, (\gamma\|X_i\|)^2\leq B}\exp\left(\frac{-B}{2\gamma^2\norm{X_i}^2}\right) \,\frac{6\gamma^3\norm{X_i}^3}{B^{\frac{1}{2}}}\\ & & +\frac{4\kappa_p^p\|\Sigma\|^{\frac{p}{2}}}{B^{\frac{p}{2}-1}} + \frac{(4+c)\gamma^4\kappa_4^4\tr{\Sigma}^2}{B},\end{eqnarray*}
with $c$ denoting the following numerical constant:
\[c:=\sum_{j=1}^{+\infty}e^{-\frac{j}{2}}\frac{6\,(j+1)^2}{j^{\frac{3}{2}}} + 4.\]
\end{lemma}

\begin{lemma}[Proof in Appendix \S \ref{appendixB}]\label{lemma:NormBounds} Let $v\in \Sd$, $x\in\R^d $, and $\gamma, B' > 0$. Assume that $\insq{x}{v}\leq B'/4$ and $\gamma^2\norm{x}^2\leq B'/4$. Then:
\[
\abs{\Gv(\insq{x}{\theta} \land B') - \left( \insq{x}{v} +\gamma^2\norm{x}^2\right)} 
\leq \exp\left(\frac{-B'}{8\gamma^2\norm{x}^2}\right) \,\frac{12\gamma^3\norm{x}^3}{(B')^{\frac{1}{2}}}.
\]
\end{lemma}

The proofs of these Lemmas are left to the Appendix, as they are quite technical and not too enlightening. 

\subsection{Concentration of the truncated process}\label{sub:proof:lem:truncatedconcentrates} We now come to the end of the section, the proof of Lemma \ref{lem:truncatedconcentrates}; that is the concentration bound on the truncated process. This will also require a technical estimate: a kind of ``counting lemma''~for the norms of vectors $\|X_i\|$. 

\begin{lemma}[Proof in Appendix  \S \ref{appendixC} ]\label{lemma:PrNorm}For any $t\in\Na$, $t\geq 3$, the event:
\begin{equation}\label{def:Norm}{\rm Norm}(t):=\bigcap_{j\geq 1}\left\{\# \left\{i\in [n]\,:\, \|X_i\|\geq e^{\frac{1}{2}}\left(\frac{n}{jt}\right)^{\frac{1}{4}}\kappa_4\sqrt{\tr{\Sigma}}\right\}\leq jr\right\}\end{equation}
has probability $\geq 1-e^{-t}/3$.\end{lemma}

Additionally, we also need another lemma; its (omitted) proof is a simple application of Bernstein's inequality.
\begin{lemma}[Proof omitted]\label{lem:simplebernstein}For any $\gamma>0$, $\beta\in (0,1)$, the following inequality holds with probability $\geq 1-\beta$: \[
\frac{1}{n}\sum_{i=1}^n \left( \left(\gamma^2\norm{X_i}^2\right) \land B -  \Ex{\left(\gamma^2\norm{X_i}^2\right) \land B}  \right) \leq \gamma^2\kappa_4^2\tr{\Sigma}\sqrt{\frac{2\, \log(2/\beta)}{n}}+\frac{2B\log(2/\beta)}{3n}.
\]\end{lemma}

\begin{proof}[Proof of Lemma \ref{lem:truncatedconcentrates}] In this proof, we set
\[\alpha_0 = \beta := \frac{\alpha}{4},\; t\geq 2,\; \gamma^{2} = \frac{1}{t}\mbox{ as in Lemma \ref{th:boundSmooth}};\]
and also, 
\begin{eqnarray*}B&:= & B_p(t) = \|\Sigma\|\,\left[\left(4\,\kappa_p^2\,\left(\frac{2000\,n}{t}\right)^{\frac{2}{p}}\right)\vee \left(512\,\kappa_4^2\,\rs\,\frac{\sqrt{n}}{t^{3/2}}\right)\right] \mbox{(from Lemma \ref{lem:counting})};\\ B'&:=& 4\,B_p(t).\end{eqnarray*}
Consider the concentration event in Lemma \ref{th:boundSmooth}, 
\[{\rm Conc}:=\left\{\widetilde{\varepsilon}_\gamma(B')
\leq  2\kappa_4^2\,\|\Sigma\|\sqrt{1+\frac{3\rs^2}{t^2}}\sqrt{\frac{t+2\log(8/\alpha)}{n}
}  
+ \frac{B'}{6\,n}(t+2\log(8/\alpha))\right\},\]
another concentration event:
\[{\rm Conc}':=\left\{\mbox{the event in Lemma \ref{lem:simplebernstein} with the choice of }\beta=\alpha/4\right\},\]
and the events ${\rm Count}(B_p(t),t)$ and ${\rm Norm}(t)$ from Lemmas \ref{lem:counting} and \ref{lemma:NormBounds} (respectively). The combination of the probability bounds from the four aforementioned lemmas implies that:
\begin{equation}\label{eq:jointoccurrence}\Pr{{\rm Conc}\cap{\rm Conc}'\cap{\rm Count}(B_p(t),t)\cap {\rm Norm}(t)} \geq 1-\frac{\alpha}{4}-\frac{\alpha}{4} - e^{-t} - \frac{e^{-t}}{3}\geq 1-\alpha.\end{equation}

This is the value of the probability appearing in the statement of the Lemma. To finish the proof, we assume in what follows that the event ${\rm Conc}\cap {\rm Conc}'\cap{\rm Count}(B_p(t),t)\cap {\rm Norm}(t)$ holds, and prove (deterministically) that the inequality (\ref{eq:lemma}) -- the claimed inequality in the statement of the Lemma -- holds whenever the four-event intersection above holds. Unfortunately, this will require some more calculations. 

It will be convenient to first state the inequality we obtain as a combination of the three events and Lemma \ref{lem:comparisonsmoothed} above. This leads to a multiline inequality: 

\begin{eqnarray}\label{eq:firstterm}\varepsilon(B)&\leq & 4\kappa_4^2\,\|\Sigma\|\sqrt{\frac{t+2\log(8/\alpha)}{n}}  
\\ \label{eq:secondterm} & & + \frac{(t+2\log(8/\alpha))\,B'}{6\,n}\\ \label{eq:partesimplebernstein} & & + \gamma^2 \kappa_4^2\tr{\Sigma}\sqrt{\frac{2\, \log(8/\alpha)}{n}}   +\frac{2B\log(8/\alpha)}{3n}\\  \label{eq:countinner} & & + \frac{4B}{n}\sup_{v\in\Sd}\#\{i\in[n]\,:\,\inner{X_i}{v}^2>B\}\\ \label{eq:countnorm} & & + \frac{4B}{n}\#\{i\in[n]\,:\,(\gamma\|X_i\|)^2>B\} \\ \label{eq:hardexponential} & & + \frac{1}{n}\sum_{i\in[n]\,:\, (\gamma\|X_i\|)^2\leq B}\exp\left(\frac{-B}{2\gamma^2\norm{X_i}^2}\right) \,\frac{6\gamma^3\norm{X_i}^3}{B^{\frac{1}{2}}}\\ \label{eq:lastline}& & +\frac{4\kappa_p^p\|\Sigma\|^{\frac{p}{2}}}{B^{\frac{p}{2}-1}} +  \frac{c\,(\gamma\kappa_4)^4\,\tr{\Sigma}^2}{B}.\end{eqnarray}
Notice that the two terms in (\ref{eq:partesimplebernstein}) come from the event ${\rm Conc}'$, which controls the first term in the RHS of the inequality in Lemma \ref{lem:comparisonsmoothed}. 

We will finish the proof by showing that each of the RHS terms in this multiline expression can be upper bounded by the RHS of (\ref{eq:lemma}),
\[C\|\Sigma\|\kappa_4^2\sqrt{\frac{t+ \log(2/\alpha)}{n}} + C\kappa_p^2\|\Sigma\|\left(\frac{t}{n}\right)^{1 - \frac{2}{p}}+C\|\Sigma\|\kappa_4^2\frac{\rs}{\sqrt{nt}},\]
for some universal choice of $C$. This means that the sum of the RHS terms can be also controlled in this same fashion, up to a change of $C$. 

This strategy becomes easier to implement if we adopt the following standard convention: in what follows, $C>0$ denotes the value of an absolute constant, which may change (i.e. be adjusted adequately) at each use. We also note that, because $t\geq \gamma^{-2}$, and using the definition of $B=B_p(t)$ above,
\begin{equation}\label{eq:boundcommonterm}\ \frac{Bt}{n} \leq 2^{11}\|\Sigma\|\,\left[\left(\kappa_p^2\left(\frac{t}{n}\right)^{1-\frac{2}{p}}\right)\vee \left(\kappa_4^2{\sqrt{\frac{1}{nt}}}\right)\right].\end{equation}

We now proceed to bound the RHS of the multiline inequality. The terms in (\ref{eq:firstterm}) is trivial, and the one in (\ref{eq:secondterm}) satisfies the desired bound because of (\ref{eq:boundcommonterm}). The terms in (\ref{eq:partesimplebernstein}) are controlled once we observe that $\gamma^2 \geq 1/\sqrt{nt}$, and apply (\ref{eq:boundcommonterm}). 

The counting term in line (\ref{eq:countinner}) is controlled by the event ${\rm Count}(B_p(t),t)$. Indeed, under this event the cardinality in that term is at most $t$, so we can use (\ref{eq:boundcommonterm}) once more.  

The terms in (\ref{eq:countnorm}) and (\ref{eq:hardexponential}) require a bit more work. First note that $\gamma^{-2}\geq 2$ under our assumptions (as $\log(8/\alpha)\geq 2$), and $t = \lceil \gamma^{-2}\rceil\geq \gamma^{-2}-1\geq \gamma^{-2}/2$. 

To continue, note that, for any index $i$, and any integer $j$, 
\[\gamma^2\|X_i\|^2\geq \frac{B}{\sqrt{j}}\Rightarrow \|X_i\|^2\geq \frac{Bt}{2\sqrt{j}}\geq \frac{155}{2}\kappa_4^2\tr{\Sigma}\sqrt{\frac{n}{tj}}.\]
Therefore, the occurrence of ${\rm Norm}(t)$ guarantees the following (with some room to spare in the constants): 
\begin{equation}\label{eq:countforB}\forall j\in\Na,\,\#\left\{i\in[n]\,:\,(\gamma\|X_i\|)^2\geq \frac{B}{\sqrt{j}}\right\}\leq jt.\end{equation}
Using the case $j=1$ of this inequality suffices to bound the term in (\ref{eq:countnorm}) by $Bt/n$, and thus by (\ref{eq:boundcommonterm}). 
For the complicated sum in (\ref{eq:hardexponential}), consider an index $i\in[n]$, and note that: 
\[\mbox{ if $j\in\Na$ and }\frac{B}{\sqrt{j+1}}< (\gamma\|X_i\|)^2\leq \frac{B}{\sqrt{j}},\,\exp\left(\frac{-B}{2\gamma^2\norm{X_i}^2}\right)\,\,\frac{6\gamma^3\norm{X_i}^3}{B^{\frac{1}{2}}}\leq \frac{B\,e^{-\frac{\sqrt{j}}{2}}}{j^{\frac{3}{4}}}.\]
Thus:
\begin{equation}\label{eq:boundhardexponential}\frac{1}{n}\sum_{i\in[n]\,:\, (\gamma\|X_i\|)^2\leq B}\exp\left(\frac{-B}{2\gamma^2\norm{X_i}^2}\right) \,\frac{6\gamma^3\norm{X_i}^3}{B^{\frac{1}{2}}}\leq \frac{1}{n}\sum_{j=1}^{+\infty}\,\,\frac{B\,e^{-\frac{\sqrt{j}}{2}}}{j^{\frac{3}{4}}}\,\ell_j,\end{equation}
where 
\[\ell_j:=\#\left\{i\in[n]\,:\,\frac{B}{\sqrt{j+1}}< (\gamma\|X_i\|)^2\leq \frac{B}{\sqrt{j}}\right\}.\]
The values $\ell_j$ are nonnegative integers. A second set of constraints comes from (\ref{eq:countforB}):
\[\ell_1\leq 2t,\, \ell_1 + \ell_2\leq 3t,\,\dots\,, \ell_1+\ell_2 + \dots +\ell_j\leq (j+1)t.\]
To bound the RHS of (\ref{eq:boundhardexponential}), imagine we try to maximize 
\[\sum_{j=1}^{+\infty}\,\frac{B\,e^{-\frac{\sqrt{j}}{2}}}{j^{\frac{3}{4}}}\,\ell_j\]
as a function of the $\ell_j$'s, subject to the above constraints. Because the $\ell_j$'s are multiplied by quantities that decrease in $j$, the best strategy to maximize the sum is to put as much mass as possible on the $\ell_j$ with the smallest indices $j$. This corresponds to $\ell_1=2t$ and $\ell_j=t$ for all $j\neq 1$. We conclude:
\[\frac{1}{n}\sum_{j=1}^{+\infty}\,\frac{B\,e^{-\frac{\sqrt{j}}{2}}}{j^{\frac{3}{4}}}\,\ell_j\leq \frac{C}{n}\,\sum_{j=1}^{+\infty}\,\frac{B\,e^{-\frac{\sqrt{j}}{2}}}{j^{\frac{-1}{4}}}\,t\leq \frac{CBt}{n},\]
which we can upper bound via (\ref{eq:boundcommonterm}). This concludes the bounding of the term in (\ref{eq:hardexponential}).

Finally, the last line terms (\ref{eq:lastline}) can be controlled by plugging in the definition of $B_p(t)$ from Lemma \ref{lem:counting} and performing some simple calculations.\end{proof}

\section{Putting everything together} \label{sec:final_est}

This section concludes the proof of the main result. This final stage of the argument consists of two steps. The first one bounds the performance of the estimators $\Est_k$ defined \S \ref{sub:defestimator} work for a range of $k$. This is basically a combination of all results from \S \ref{sub:defestimator} onwards. 

\begin{lemma}[Proven in \S \ref{sub:proof:lem:manyktruncate}]\label{lem:manyktruncate} Consider a constant $c>0$ such that: 
\[k_0:=\lfloor \eta n \rfloor + \lceil c\eta n + \rs + \log(32/3\alpha)\rceil<n.\]
Also fix $p\geq 4$. Then the following holds with probability $\geq 1-\alpha/2$:  
\[\bigcap_{k=k_0}^{n-1}\left\{\|\Est_k - \Sigma\|\leq C\|\Sigma\|\kappa_4^2\sqrt{\frac{\rs + \log(2/\alpha) + (k-k_0)}{n}} + C_c\kappa_p^2\|\Sigma\|\left(\frac{k}{n}\right)^{1 - \frac{2}{p}}\right\},\]
where $C_c>0$ depending only on $c$; here, $\Est_k$ is the estimator obtained in Proposition \ref{prop:estimatorisdefined}.\end{lemma}

The second step will be the final construction of the estimator, in \S \ref{sub:finalfinal}. 

\subsection{Proof of Lemma \ref{lem:manyktruncate}} \label{sub:proof:lem:manyktruncate}

\begin{proof}[Proof of Lemma \ref{lem:manyktruncate}] Recall that the error \[\|\Est_k-\Sigma\|\leq 2\sup_{v\in \Sd}|\est_k(v)-\inner{v}{\Sigma v}|\] (cf. Proposition \ref{prop:estimatorisdefined}). Thus the lemma follows from the following observation: for any $k\geq k_0$, the following holds with probability $\geq 1-2^{-k+k_0-2}\alpha$ 
\[\sup_{v\in \Sd}|\est_k(v) - \inner{v}{\Sigma v}|\leq C\|\Sigma\|\kappa_4^2\sqrt{\frac{\rs + \log(2/\alpha) + (k-k_0)}{n}} + C\kappa_p^2\|\Sigma\|\left(\frac{k}{n}\right)^{1 - \frac{2}{p}},\]
with $C>0$ universal. Indeed, the Lemma follows from the above from a union bound. 

To prove the above observation, apply Lemma \ref{lem:truncatedconcentrates} with 
\[\alpha\mbox{ replaced by }\alpha_k:=\frac{\alpha}{2^{k-k_0+2}}\mbox{ and }t:=\lceil c\eta n + \rs + \log(8/3\alpha_k)\rceil,\]
noting that $\lfloor \eta n\rfloor + t\leq k$. It follows that, with probability $\geq 1-\alpha_k$, 
\[\varepsilon(B_p(t))\leq C\|\Sigma\|\kappa_4^2\sqrt{\frac{\rs + \log(2/\alpha) + (k-k_0)}{n}} + C\kappa_p^2\|\Sigma\|\left(\frac{k}{n}\right)^{1 - \frac{2}{p}}.\]
The proof of Lemma \ref{lem:truncatedconcentrates} obtains the above inequality as a consequence of a certain event holding, which is contained in ${\rm Count}(B_p(t),t)$. Thus we can assume that the previous display {\em and} ${\rm Count}(B_p(t),t)$ {\em both} hold with probability $\geq 1-\alpha_k$. In that case Proposition \ref{prop:estimatortruncatedempirical} (with $B=B_p(t)$ and $s=t$) gives: 
\[\sup_{v\in \Sd}|\est_k(v) - \inner{v}{\Sigma v}|\leq \varepsilon(B_p(t)) + \frac{B_p(t)\,k}{n} + \frac{\kappa_p^p\tr{\Sigma}^{\frac{p}{2}}}{B^{\frac{p}{2}-1}}.\]
We finish the proof by using $CB_p(t)\geq\kappa_p^2 \|\Sigma\|(n/t)^{1/p}$; noting that $k\leq C_ct$ for some $C_c>0$ depending on $c$ only; and bounding $B_p(t)\,t/n$ as in (\ref{eq:boundcommonterm}).\end{proof}  

\subsection{Proof of the main result}\label{sub:finalfinal}

\begin{proof}[Proof of Theorem \ref{thm:main}] The basic idea of this final argument is that, once Lemma \ref{lem:manyktruncate} is in place, one ``wins the game'' by finding a $\widehat{k}$ similar to $k_0$, for which what we need is an estimator for $\rs$.

Let us give the details, allowing in this proof for absolute constants $C$, $D$ whole value may change at each appearance. Notice that the dependence on the parameter $c$ disappears: we simply take $c=1$ and require that $\eta\,\kappa_4^4$ is suitably small (which is part of our assumption).

Corollary \ref{cor:trace} presents a trimmed mean estimator for the trace. Specifically, that result implies the following: if $\alpha\in (0,1)$ satisfies:
\[k:=\lfloor \eta n\rfloor + \lceil \eta n + \log(4/\alpha)\rceil<n\]
and moreover $\eta\leq (1/4C\,\kappa_4^2)^{2}$ and $n\geq C^2\kappa_4^4\log(4/\alpha)$, then the estimator $\Trest$ from that corollary satisfies
\begin{equation}\label{eq:traceisgood}\Pr{\frac{\tr{\Sigma}}{2}\leq \Trest\leq \frac{3\tr{\Sigma}}{2}}\geq 1 - \frac{\alpha}{2}.\end{equation}
Furthermore, by Lemma \ref{lem:manyktruncate}, with probability $1-\alpha/2,$
\begin{equation}\label{eq:manyktrunc}
\bigcap_{k=k_0}^{n-1}\left\{\|\Est_k - \Sigma\|\leq C\|\Sigma\|\kappa_4^2\sqrt{\frac{\rs + \log(2/\alpha) + (k-k_0)}{n}} + C_c\kappa_p^2\|\Sigma\|\left(\frac{k}{n}\right)^{1 - \frac{2}{p}}\right\}.\end{equation}
{\em Assume temporarily the events in (\ref{eq:traceisgood}) and (\ref{eq:manyktrunc}) both hold.} 
In that case, notice that there exists $D>0$ depending only on $c>0$ such that, if $n\geq D\kappa_p^p(\rs + \log(2/\alpha))$, then $k^*:=\lfloor n/D\rfloor$ satisfies: 
\[\|\Est_{k^*}-\Sigma\|\leq \frac{\|\Sigma\|}{2}, \text{ \,\, so that\,\,  } \frac{\|\Sigma\|}{2}\leq \|\Est_{k^*}\|\leq \frac{3\|\Sigma\|}{2}.\]
Therefore, 
\begin{equation}\label{eq:rsestimate}\frac{\rs}{3}\leq \frac{\Trest}{\Est_{k^*}}\leq 3\rs.\end{equation}
Now set:
\[\widehat{k}:=\lfloor \eta n \rfloor + \left\lceil \eta n + \frac{3\Trest}{\|\Est_{k^*}\|} + \log(32/3\alpha)\right\rceil.\]
In our current event, $k_0\leq \widehat{k}<n$ by (\ref{eq:rsestimate}) and the assumption in our theorem. Moreover, $\widehat{k}\leq k_0 + D'\tr{\Sigma}$ for some absolute constant $D'>0$.

\noindent{\bf End of the argument:} We conclude from all of the above that if $D$ is suitably large, and $k_*$, $\widehat{k}$ are defined as above, then: 
\begin{equation*}\|\Est_{\widehat{k}} - \Sigma\|\leq C\|\Sigma\|\kappa_4^2\sqrt{\frac{\rs + \log(2/\alpha)}{n}} + C_c\kappa_p^2\|\Sigma\|\left(\frac{k}{n}\right)^{1 - \frac{2}{p}}\end{equation*}
holds whenever the events in (\ref{eq:traceisgood}) and (\ref{eq:manyktrunc}) both hold. In particular, the probability of the last display holding is at least $\geq 1-\alpha$. This is what we want for $\eta n \geq \log(2/\alpha)$ because in that case $k\leq C \eta n$, and
\[ C\kappa_p^2\|\Sigma\|\left(\frac{k}{n}\right)^{1 - \frac{2}{p}}\leq {C}\kappa_p^2\|\Sigma\|\eta^{1-\frac{2}{p}},\]
for another absolute constant $C>0$. Otherwise, when $\eta n \leq \log(2/\alpha)$, then $k\leq C\log(2/\alpha)$ and we may take $p=4$ above to obtain the desired bound.\end{proof}

\appendix
The Appendix is divided into four main parts. The first three parts prove Lemmas \ref{lem:comparisonsmoothed}, \ref{lemma:NormBounds} and  \ref{lemma:PrNorm} from the main text. In the fourth part, we introduce a variation of the main result, Theorem 1.3 which relaxes the moment conditions.

\section{Proof of Lemma \ref{lem:comparisonsmoothed}}\label{appendixA}

\begin{proof}[Proof of Lemma \ref{lem:comparisonsmoothed}] Introduce a function:
\[f(\theta,x):=\inner{x}{\theta}^2\wedge B + (\gamma\|x\|)^2\wedge B\,\,\,((\theta,x)\in \R^d\times \R^d),\]
and define an intermediate object between $\varepsilon(B)$ and $\varepsilon_\gamma(4B)$:
\[\varepsilon^*(B):=\sup_{v\in\Sd}\left|\frac{1}{n}\sum_{i=1}^n(f(v,X_i) - \Ex{f(v,X_i)})\right|.\]

Recalling the definition of $\varepsilon(B)$ from equation (\ref{eq:definitionvarepsilon}) of the main text, we obtain
\begin{equation}\label{eq:comparestep1}\varepsilon(B)\leq  \varepsilon^*(B) + \left|\frac{1}{n}\sum_{i=1}^n(\gamma\|X_i\|)^2\wedge B - (\gamma\|X_1\|)^2\wedge B \right|. \end{equation}

We now compare $\varepsilon^*(B)$ and $\widetilde{\varepsilon}_\gamma(4B)$, recalling the definition of the latter from equation (\ref{eq:defsmoothedtruncated}) of the main text. Each of these quantities is a supremum over $v\in \Sd$ of certain sums. A term-by-term comparison of the sums leads to the following bound:
\begin{eqnarray}\nonumber \varepsilon^*(B) - \widetilde{\varepsilon}_\gamma(4B) & \leq & \sup_{v\in\Sd}\frac{1}{n}\sum_{i=1}^n\left|f(v,X_i) - \Gamma_{0,\gamma}[\inner{X_i}{\theta}^2\wedge (4B)]\right| \\ \nonumber & & + \sup_{v\in\Sd}\left|\Ex{f(v,X_1)} - \Ex{\Gamma_{0,\gamma}\inner{X_1}{\theta}^2\wedge (4B)}\right| \\ \label{eq:oneplustocomparison} &=:& (I) + (II).\end{eqnarray}
We deal with $(I)$ via Lemma \ref{lemma:NormBounds}. Break the different indices $i\in[n]$ into two groups. Group (a) consists of the indices for which $\gamma\|X_i\|^2> B$ or $\inner{X_i}{v}^2>B$, we note that $0\leq f\leq 2B$ and use the simple bound:
\[\left|f(v,X_i) - \Gamma_{0,\gamma}[\inner{X_i}{\theta}^2\wedge (4B)]\right|\leq 4B.\]
Group (b) consists of the remaining indices. To these terms one can apply Lemma \ref{lemma:NormBounds} with $B'=4B$ and obtain:
\[\left|f(v,X_i) - \Gamma_{0,\gamma}[\inner{X_i}{\theta}^2\wedge (4B)]\right|\leq \exp\left(\frac{-B}{2\gamma^2\norm{X_i}^2}\right) \,\frac{6\gamma^3\norm{X_i}^3}{B^{\frac{1}{2}}}.\]
Adding up the contributions of these two groups, we arrive at:
\begin{eqnarray}\nonumber (I)&\leq & \frac{4B}{n}\sup_{v\in\Sd}\#\{i\in[n]\,:\,\inner{X_i}{v}^2>B\}+ \frac{4B}{n}\#\{i\in[n]\,:\,(\gamma\|X_i\|)^2>B\} \\ \label{eq:boundIv}& & + \frac{1}{n}\sum_{i\in[n]\,:\, (\gamma\|X_i\|)^2\leq B}\exp\left(\frac{-B}{2\gamma^2\norm{X_i}^2}\right) \,\frac{6\gamma^3\norm{X_i}^3}{B^{\frac{1}{2}}}.\end{eqnarray}

The control of $(II)$ is a bit messier. The reasoning used above may be used ``inside the expectation" to obtain that. For any unit vector $v\in\Sd$,

\begin{eqnarray*}\left|\Ex{f(v,X_1)} - \Ex{\Gamma_{0,\gamma}\inner{X_1}{\theta}^2\wedge (4B)}\right| &\leq &  4B\,\Pr{\inner{X_1}{v}^2\geq B\mbox{ or }(\gamma^2\|X_i\|)^2> B} \\ &  +& \Ex{\Ind{(\gamma\|X_1\|)^2\leq B}\,\exp\left(\frac{-B}{2\gamma^2\norm{X_1}^2}\right) \,\frac{6\gamma^3\norm{X_1}^3}{B^{\frac{1}{2}}}}.\end{eqnarray*}

The probability term can be bounded via the moment assumptions (including Proposition \ref{prop:Minkowski}):
\[ 4B\,\Pr{\inner{X_1}{v}^2> B\mbox{ or }(\gamma^2\|X_i\|)^2>B}\leq \frac{4\kappa_p^p\|\Sigma\|^{\frac{p}{2}}}{B^{\frac{p}{2}-1}} + \frac{4\gamma^4\kappa_4^4\tr{\Sigma}^2}{B}.\]

As for the expectation, we break the event appearing inside it into a disjoint union: \[\{(\gamma\|X_1\|)^2\leq B\} =\bigcup_{j=1}^{+\infty}E_j\mbox{ where }E_j:=\left\{\frac{B}{j+1}< (\gamma\|X_1\|)^2\leq \frac{B}{j}\right\}.\]
Therefore, 
\[\Ex{\Ind{(\gamma\|X_1\|)^2\leq B}\,\exp\left(\frac{-B}{2\gamma^2\norm{X_1}^2}\right) \,\frac{6\gamma^3\norm{X_1}^3}{B^{\frac{1}{2}}}}=\sum_{j=1}^{+\infty}\Ex{\Ind{E_j}\,\exp\left(\frac{-B}{2\gamma^2\norm{X_1}^2}\right) \,\frac{6\gamma^3\norm{X_1}^3}{B^{\frac{1}{2}}}}.\]

For each $E_j$, 
\[\Pr{E_j}\leq \Pr{(\gamma\|X_1\|)^2>\frac{B}{j+1}}\leq \frac{\gamma^4\,(j+1)^2\,\Ex{\|X_1\|^4}}{B^2}\leq \frac{\kappa_4^4\gamma^4\,(j+1)^2\tr{\Sigma}^2}{B^2}.\]
Moreover, inside each $E_j$ we have an upper bound on the integrand:  
\[\gamma^2\|X_1\|^2\leq \frac{B}{j}\Rightarrow  \exp\left(\frac{-B}{2\gamma^2\norm{X_1}^2}\right) \,\frac{6\gamma^3\norm{X_1}^3}{B^{\frac{1}{2}}}\leq \exp\left(-\frac{j}{2}\right)\frac{6B}{j^{\frac{3}{2}}}.\]
Putting these bounds together,
\begin{equation}\label{eq:boundIIcompare}(II) \leq  \frac{4\kappa_p^p\|\Sigma\|^{\frac{p}{2}}}{B^{\frac{p}{2}-1}} + \frac{4\gamma^4\kappa_4^4\tr{\Sigma}^2}{B} +  \frac{(\gamma\kappa_4)^4\,\tr{\Sigma}^2}{B}\sum_{j=1}^{+\infty}e^{-\frac{j}{2}}\frac{6\,(j+1)^2}{j^{\frac{3}{2}}}.\end{equation}
Looking back at (\ref{eq:comparestep1}), (\ref{eq:oneplustocomparison}) and the bounds obtained in (\ref{eq:boundIv}) and (\ref{eq:boundIIcompare}), we can see all the terms in the RHS of the bound claimed in the statement of the Lemma. This finishes the proof.\end{proof}

\section{Proof of Lemma \ref{lemma:NormBounds}}\label{appendixB}
\begin{proof}[Proof of Lemma \ref{lemma:NormBounds}]
For simplicity, let $m:=\inner{X_i}{v}$ and $s:=\gamma^2\norm{X_i}^2$. Let $N$ be a normal random variable with mean $m$ and variance $s$. Then:
\[\Gv(\inner{x}{\theta}^2 \land B') = \Ex{N^2 \land B'}
\]
whereas $\Ex{N^2}= m+s$. Therefore, our goal is to bound: 
\[|\Ex{N^2\wedge B'} - \Ex{N^2}| = \Ex{(N^2-B')_+}= \frac{\int_{\sqrt{B'}}^{+\infty}\left( t^2-B'\right) \left(e^{-\frac{(t-m)^2}{2s}}+e^{-\frac{(t+m)^2}{2s}} \right)dt}{\sqrt{2\,\pi\,s}}.\]
We will perform the change of variables $t\to u+\sqrt{B'}$. Since we only consider $t\geq \sqrt{B'}$ in the integral, the variable $u$ will be always nonnegative. Moreover, $\sqrt{B'}/2\geq |m|$ by assumption, and so: \[(t\pm m)^2 = u^2 + (\sqrt{B'}\pm m)^2 + 2u(\sqrt{B'}\pm m)\geq \frac{B'}{4} + u\sqrt{B'}.\]
Since we also have and $t^2-B'=2u\sqrt{B'}+u^2,$ we can go back to the integrals and obtain:
\[ 
\Ex{(N^2-B')_+} \leq \frac{2\exp\left(-\frac{B'}{8s}\right)}{\sqrt{2\pi s}}\int_{0}^{+\infty}\left(2\sqrt{B'}u+u^2 \right)\exp\left(-\frac{\sqrt{B'}u}{2s}\right)\,du.
\]
One may use the formula:
\[\forall\eta>0, \forall a\in\Na:\,\,\, \int_{0}^{+\infty}u^ae^{-\eta u} du =\frac{a!}{\eta^{a+1}}, 
\]
corresponding to the moments of the exponential distribution, and the choice $\eta=\sqrt{B'}/2s$ to obtain:
\[ 
\int_{0}^{+\infty}\left(2\sqrt{B'}u+u^2 \right)\exp\left(-\frac{\sqrt{B'}u}{2s}\right)\,du =  \frac{8s^2}{\sqrt{B'}} + \frac{16s^3}{(B')^\frac{3}{2}}.
\]
Bounding $\sqrt{2\pi}\geq 2$ in the denominator, we 
conclude:
\[\Ex{(N^2-B')_+} \leq \exp\left(-\frac{B'}{8s}\right)\,\left(\frac{8s^{\frac{3}{2}}}{\sqrt{B'}} + \frac{16s^{\frac{5}{2}}}{(B')^\frac{3}{2}}\right),\]
and our final bound is obtained when we note that:
\[s\leq \frac{B'}{4}\Rightarrow \frac{s^{\frac{5}{2}}}{(B')^{\frac{3}{2}}}\leq \frac{s^{\frac{3}{2}}}{4(B')^{\frac{1}{2}}}.\]
\end{proof}

\section{Proof of Lemma \ref{lemma:PrNorm}}\label{proof:lemma:PrNorm}\label{appendixC}
\begin{proof}The first step of the proof i
s to note that for any $r \in \Na \setminus \{ 1\},$
\begin{equation}\label{claim:Norm}
\Pr{\# \left\{i\in [n]\,:\, \|X_i\|\geq e^{\frac{1}{2}}\left(\frac{n}{t}\right)^{\frac{1}{2}}\kappa_4\sqrt{\tr{\Sigma}}\right\}> r} \leq e^{-r-2}.
\end{equation}
In fact, this is Claim \ref{claim:count} in the main text with $Z_i:=\|X_i\|^2$, $q=2$ and $t=r$ (we use Proposition \ref{prop:Minkowski} to bound $\Ex{Z_i^2}$). A union bound for $r=t,2t, 3t,\dots$ gives
\[\Pr{\bigcup_{j=1}^\infty \left\{\# \left\{i\in [n]\,:\, \|X_i\|\geq \sqrt{e}\,\left(\frac{n}{jt}\right)^{\frac{1}{2}}\kappa_4\sqrt{\tr{\Sigma}}\right\}\geq jt\right\}} \leq \sum_{j\geq 1} e^{-jt-2}=\frac{e^{-t}}{e^2-e^{2-t}}.
\]
The RHS is bounded by $e^{-t}/3$ for $t\geq 3$. Thus, the lemma follows. 
\end{proof}

\section{Weak moment bounds assumption} \label{appendixD}

In this section, we explain briefly how our proof can be adapted to deal with even weaker moment assumptions than our main result. More specifically, we require the following. 

\begin{assumption}\label{assum2}
$X_1,\dots,X_n$ are i.i.d. copies of a random element $X$ of $\R^d$ satisfying $\Ex{\|X\|^q}<+\infty$ for some $q\geq 2$. We assume $\Ex{X}=0$, that the covariance $\Sigma$ of $X$ is non-null, and set 
\[\kappa_q:=\sup\limits_{v\in\R^d,\,\inner{v}{\Sigma v}=1}\Ex{|\inner{X}{v}|^q}^{\frac{1}{q}}.\]
Let $Y_1,\dots,Y_n$ be random elements of $\R^d$ that satisfy:
\[\#\{i\in [n]\,:Y_i\neq X_i\}\leq \eta n\]
for some $\eta\in [0,1)$.
\end{assumption}

\begin{proposition}\label{th:Weak_main}
 Assume $\alpha\in (0,1)$, $n\in\Na$ and a contamination parameter $\eta\in [0,1/2)$. Then there is an estimator, depending on $\alpha$, $\eta$, and $n$, such that, whenever Assumption \ref{assum2} holds, $\eta\leq 1/C\kappa_4^4$, and $n\geq C\,(\rs + \log(2/\alpha)),$ then with probability $\geq 1-\alpha$:
\[\normop{\Est_{\star}(Y_1,\dots,Y_n) - \Sigma}\leq C\kappa_q^2\normop{\Sigma}\left(\frac{\rs + \log(2/\alpha)}{n}\right)^{1-\frac{2}{q}}+C\kappa_q^2\normop{\Sigma}\eta^{1-\frac{2}{q}},\] 
where $C>0$ is an absolute constant. 
\end{proposition}

\begin{proof}[Proof sketch]
Once we establish control over the truncated empirical processes, as in Lemma \ref{lemma:higermomentsBq} below, the proof follows analogously to that of the main Theorem \ref{thm:main}.
\end{proof}

The following lemma establishes an upper bound for the smoothed truncated empirical process within the specified context. 

\begin{lemma}\label{lemma:higermoments} 
Consider $B'>0$, $t > 0$, $\gamma^{-2}:=t$, and $2\leq q \leq 4$. Then
\[
\widetilde{\varepsilon}_\gamma(B')
\leq C\frac{B'\,t}{n}+C \nu_{q}^{2}\left(\frac{t}{n}\right)^{1-\frac{2}{q}},
\]
 with probability at least $1-e^{-t}$, where
 \[\nu_{q}^{q}:=\sup_{||v||=1} \Ex{|\ip{X_i}{v}^{q}|}+C\gamma^q \kappa_q^q\,\Ex{||X_i||^{q}}.\] 
\end{lemma} 

 \begin{proof}
 We apply Proposition \ref{th:PacBernstein} to \[Z_i(\theta)= \pm\frac{1}{n} (\insq{X_i}{\theta}\land B' - \Ex{\insq{X_i}{\theta}\land B'})_{i\in[n], \theta\in\R^d}\] as in the proof of Lemma \ref{th:boundSmooth}. Then, it holds for all $v\in\Sd$ with probability at least $1-e^{-t}$,
\begin{eqnarray*} 
\frac{1}{n}\sum_{i=1}^n\Gv \left( \insq{X_i}{\theta} \land B'-\Ex{ \insq{X_i}{\theta} \land B'}\right) && \leq \\ \sqrt{\sup_{v\in\Sd}\sum_{i=1}^n\Gv  \Var{\insq{X_i}{\theta}\land B'}\frac{(\gamma^{-2}+2t)}{n}}   + \frac{(\gamma^{-2}+2t)B'}{6\,n}.
\end{eqnarray*}
 Then, we need to upper bound the first term. For this, observe that 
 \begin{eqnarray*}
\Ex{\Gamma_{v,\gamma}\left(\ip{X_i}{\theta}^2\land B'\right)^2} &\leq& \Ex{\Gamma_{v,\gamma}\left(\ip{X_i}{\theta}^2\land B'\right)}^2\\ &\leq& 
B'^{2-\frac{q}{2}}\Ex{\Gamma_{v,\gamma}|\ip{X_i}{\theta}|^{q}}\\
&\leq& B'^{2-\frac{q}{2}}\nu_q^q.
 \end{eqnarray*}
The second inequality comes from our hypothesis over $q$. It follows from above that
 \begin{eqnarray*}
\sqrt{\frac{\gamma^{-2}+2t}{n}\sup_{||v||=1}\Ex{\Gamma_{v,\gamma}\left(\ip{X_i}{\theta}^2\land B'\right)^2}}&\leq& C\,B'^{1-\frac{q}{4}} \left(\frac{\gamma^{-2}+2t}{n}\right)^{\frac{1}{2}}\nu_{q}^\frac{q}{2}\\
&\leq& C\left(\frac{B'\,t}{n}\right)^{1-\frac{q}{4}}\left(\nu_{q}^{2}\left(\frac{t}{n}\right)^{1-\frac{2}{q}}\right)^\frac{q}{4}\\
&\leq& C \frac{B'\,t}{n}+\nu_{q}^2\left(\frac{t}{n}\right)^{1-\frac{2}{q}}.
 \end{eqnarray*} 
 \end{proof}

To conclude, we use Lemma \ref{lemma:higermoments} above to control the empirical truncated process.

 \begin{lemma}\label{lemma:higermomentsBq} 
Make Assumption \ref{assum2}. Take $\alpha\in (0,1)$ and $t\in\Na$ with $t\geq 2$. There exists a universal $C_0>0$ such that, if one chooses 
\begin{equation}
B_q(t)\leq C_0\kappa_{q}^2\normop{\Sigma}\left(\frac{n}{t}\right)^{\frac{2}{q}}\left(1+\frac{\rs}{t}\right), \label{eq:Bqt}  \end{equation} then the following holds with probability $\geq 1-\alpha$:
\begin{equation}
\varepsilon(B_q(t))\leq C\kappa_q^2\normop{\Sigma}\left(1+\frac{\rs}{t}\right)\left(\frac{t}{n}\right)^{1-\frac{2}{q}}+ C\normop{\Sigma}\left(\frac{t}{n}\right)^{1-\frac{2}{q}}
\end{equation}
where $C>0$ is a universal constant independent of all other problem parameters.
 \end{lemma} 

\begin{proof}
Lemma \ref{lem:comparisonsmoothed} provides a general bound for the difference between the smoothed and unsmoothed truncated empirical processes, for all values of $B$ and $\gamma$ greater than zero. To address our moments assumption, we first establish an upper bound for the initial term on the RHS using the Bernstein inequality as follows.

\begin{eqnarray*}
\left|\frac{1}{n}\sum_{i=1}^n\left\{(\gamma\|X_i\|)^2\wedge B_q(t) - \Ex{(\gamma\|X_1\|)^2\wedge B_q(t) }\right\}\right|&\leq& C\frac{B_q(t)\,t}{n}+ C \frac{B_q(t)\log(1/\alpha)}{6n} \\ &+& C \kappa_q\left(\frac{n}{t}\right)^{2/q}\frac{\tr{\Sigma}}{n} \left(\frac{\log(1/\alpha)}{t}\right)^{2/q}, \end{eqnarray*}
with probability at least $1-\alpha$. 
It follows from the fact that:
\[
\Ex{(\gamma||X_1||)^2\wedge B_q(t)}^2\leq B_q(t)^{2-q/2}\gamma^q\Ex{X_1}^q\leq \kappa_q^{q/2}B_q(t)^{2-q/2}\gamma^q\tr{\Sigma}^{q/2}.
\]

Next, Lemma \ref{lemma:higermoments} bounds the smoothed truncated empirical processes. Utilizing our selected $B_q(t)$ hypothesis and recalling that
\[\nu_q^2\leq c\kappa_p^2\normop{\Sigma}\left(1+\kappa_q \frac{\rs}{t}\right),\] we can deduce that:
\[
\widetilde{\varepsilon}_\gamma(B_q(t))
\leq C \kappa_q^2\normop{\Sigma}\left(\frac{t}{n}\right)^{1-\frac{2}{q}} \left(1+\frac{\rs}{t}\right).
\]
Finally, the last term in Lemma \ref{lem:comparisonsmoothed} can be bounded as following. 
\[\frac{4\kappa_q^q\normop{\Sigma}^{q/2}}{B_q(t)^{q/2-1}}+\frac{4\kappa_q^q\tr{\Sigma}^{q/2}}{B_q(t)^{q/2-1}}.\]
By substituting the expression for $B_q(t)$ as given in (\ref{eq:Bqt}), we arrive at the desired result.
\end{proof}

\bibliography{references}
\bibliographystyle{apalike}
\end{document}